\documentclass[11pt]{amsart}
\input{commands.tex}
\usepackage{tikz-cd}
\usepackage{enumerate}
\usepackage[shortlabels]{enumitem}
\usepackage{microtype}
\usepackage{amsfonts}
\usepackage{longtable}

\usepackage[normalem]{ulem}

\usepackage{xcolor}
\definecolor{darkgreen}{rgb}{0,0.30,0} 
\definecolor{darkred}{rgb}{0.75,0,0}
\definecolor{darkblue}{rgb}{0,0,0.6} 
\definecolor{lightblue}{RGB}{179, 230, 255}
\usepackage{amsthm,thmtools}
\usepackage[pdfborder=0,pagebackref,colorlinks,citecolor=darkgreen,linkcolor=darkgreen,urlcolor=darkblue]{hyperref}
\usepackage{cleveref}

\def\makeautorefname#1#2{\expandafter\def\csname#1autorefname\endcsname{#2}}
\makeautorefname{eqn}{Equation}%
\makeautorefname{sec}{Section}%
\makeautorefname{subsec}{Subsection}%
\makeautorefname{footnote}{footnote}%
\makeautorefname{item}{item}%
\makeautorefname{figure}{Figure}%
\makeautorefname{table}{Table}%
\makeautorefname{part}{Part}%
\makeautorefname{app}{Appendix}%
\makeautorefname{cla}{claim}%
\makeautorefname{conj}{conjecture}%
\makeautorefname{cor}{corollary}%
\makeautorefname{cex}{counterexample}%
\makeautorefname{cexs}{counterexamples}%
\makeautorefname{dig}{digression}%
\makeautorefname{disc}{discussion}%
\makeautorefname{def}{definition}%
\makeautorefname{ex}{example}%
\makeautorefname{exs}{examples}%
\makeautorefname{fac}{fact}%
\makeautorefname{intu}{intuition}%
\makeautorefname{lem}{lemma}%
\makeautorefname{nota}{notation}%
\makeautorefname{note}{note}%
\makeautorefname{prop}{proposition}%
\makeautorefname{rmk}{remark}%
\makeautorefname{term}{terminology}%
\makeautorefname{thm}{theorem}%
\makeautorefname{upsh}{upshot}%
\theoremstyle{definition}
\newtheorem{theorem}{Theorem}[subsection]

\newtheorem{corollary}[theorem]{Corollary}

\newtheorem{definition}[theorem]{Definition}

\newtheorem{example}[theorem]{Example}

\newtheorem{lemma}[theorem]{Lemma}
\newtheorem{notation}[theorem]{Notation}

\newtheorem{proposition}[theorem]{Proposition}
\newtheorem{remark}[theorem]{Remark}
\newtheorem{terminology}[theorem]{Terminology}

\makeatletter
\let\c@corollary=\c@theorem
\let\c@proposition=\c@theorem
\let\c@lemma=\c@theorem
\let\c@conjecture=\c@theorem
\let\c@definition=\c@theorem
\let\c@example=\c@theorem
\let\c@remark=\c@theorem
\let\c@notation=\c@theorem
\let\c@equation\c@theorem
\makeatother

\author{Thomas Brazelton}
\title{Homotopy Mackey functors of equivariant algebraic $K$-theory}
\providecommand{\theauthor}{Thomas Brazelton}
\providecommand{\thetitle}{Homotopy Mackey functors of equivariant algebraic $K$-theory}
\usepackage{fancyhdr}
\usepackage[color=lightblue,textsize=tiny]{todonotes}

\providecommand{\Fun}{\text{Fun}}
\providecommand{\Aut}{\text{Aut}}
\providecommand{\End}{\text{End}}

\providecommand{\sh}{\text{sh}}
\providecommand{\Span}{\text{Span}}
\providecommand{\THH}{\text{THH}}

\providecommand{\Tau}{\mathcal{T}}

\providecommand{\Wald}{\text{Wald}}
\renewcommand{\Sp}{\texttt{Sp}}
\providecommand{\ExCat}{\texttt{ExCat}}
\providecommand{\TP}{\text{TP}}
\providecommand{\TR}{\text{TR}}
\providecommand{\TC}{\text{TC}}

\newcommand{\twR}[1]{R_{\theta}\left[#1\right]}

\providecommand{\Ring}{\texttt{Ring}}
\providecommand{\Ab}{\texttt{Ab}}
\let\und\underline

\renewcommand{\O}{\mathcal{O}}

\definecolor{seagreen}{RGB}{46,139,87}

\newlength{\storeparskip}
\begin{document}

\maketitle
\begin{abstract}
Given a finite group $G$ acting on a ring $R$, Merling constructed an equivariant algebraic $K$-theory $G$-spectrum, and work of Malkiewich and Merling, as well as work of Barwick, provides an interpretation of this construction as a spectral Mackey functor. This construction is powerful, but highly categorical; as a result the Mackey functors comprising the homotopy are not obvious from the construction and have therefore not yet been calculated. In this work, we provide a computation of the homotopy Mackey functors of equivariant algebraic $K$-theory in terms of a purely algebraic construction. In particular, we construct Mackey functors out of the $n$th algebraic $K$-groups of group rings whose multiplication is twisted by the group action. Restrictions and transfers for these functors admit a tractable algebraic description in that they arise from restriction and extension of scalars along module categories of twisted group rings. In the case where the group action is trivial, our construction recovers work of Dress and Kuku from the 1980's which constructs Mackey functors out of the algebraic $K$-theory of group rings. We develop many families of examples of Mackey functors, both new and old, including $K$-theory of endomorphism rings, the $K$-theory of fixed subrings of Galois extensions, and (topological) Hochschild homology of twisted group rings.
\end{abstract}

\setcounter{tocdepth}{1}
\tableofcontents{}

\setlength{\parskip}{.2em}

\section{Introduction}

Algebraic $K$-theory is a theory which encodes profound algebraic invariants of rings, revealing beautiful connections and patterns. The roots of such a theory run deep throughout mathematics, and algebraic $K$-theory can now be found in almost every corner of algebra. For example, the class number formula generalizes to a statement about the torsion of algebraic $K$-groups of number fields. A complete understanding of the $K$-groups of the integers would resolve the Kummer--Vandiver conjecture, formulated over a century and a half ago at the time of writing, and which stands as one of the great unsolved problems in number theory. However, the power and versatility of algebraic $K$-theory comes at the cost of its formulation relying on sophisticated categorical machinery.

Following the development of algebraic $K$-theory, there was a surge of interest in defining an equivariant version. In the 1980's, Fiedorowicz, Hauschild, and May gave a plus construction yielding the first topological space encoding equivariant algebraic $K$-theory for rings \cite{FHM}, and Kuku, together with Dress, developed a $Q$-construction for equivariant $K$-theory for exact categories \cite{Kuku-equivar,DK-cartan,DK-convenient}. However both of these definitions only worked in the case where the rings in question were equipped with a trivial $G$-action.

The slow progress in the development of a more robust equivariant algebraic $K$-theory can be explained by the lack of equivariant homotopical tools available at the time. The recent development of equivariant infinite loop space theory by Guillou, May, Merling, and Osorno \cite{inf-loop,Segal-inf-loop} was a major step in the direction towards laying the groundwork for equivariant algebraic $K$-theory. Using these new techniques, Merling was able to build a genuine equivariant plus construction for algebraic $K$-theory for rings equipped with a group action \cite{Mona}.

While classical algebraic $K$-theory produces a $K$-theory spectrum whose homotopy groups are the algebraic $K$-groups of a ring, equivariant algebraic $K$-theory produces a genuine $G$-spectrum, whose homotopy in each degree is a strictly richer algebraic object called a \textit{Mackey functor}.

\[ \begin{tikzcd}
    \Ring\ar[rr,"K_i" above]\ar[dr,"K" below left] &  & \Ab\\
     & \Sp\ar[ur,"\pi_i" below right] & 
\end{tikzcd} \quad \begin{tikzcd}
    G\text{-}\Ring\ar[rr,"\underline{K}_i" above]\ar[dr,"\mathbf{K}_G" below left] & & \substack{\text{Mackey}\\ \text{functors}}\\
     & G\text{-}\Sp\ar[ur,"\und{\pi}_i" below right]
\end{tikzcd} \]

Mackey functors, as initially formulated by Dress \cite{Dress1,Dress2} and Green \cite{Green}, are a collection of abelian groups indexed over subgroups of $G$, equipped with restriction and transfer maps subject to various axioms (\autoref{def:Mackey-functor}). They may equivalently be defined as modules over an additive category called the \textit{Burnside category}, denoted $B_G$. Following the adage that spectra are the homotopy theorist's abelian groups, one can envision an analogue of Mackey functors which arise as modules over a spectrally enriched version of the Burnside category. Such a construction is called a \textit{spectral Mackey functor}, which have been explored in their greatest generality by Barwick \cite{SMF1} together with Glasman and Shah \cite{BGS}. A celebrated result of Guillou and May demonstrates that spectral Mackey functors are equivalent to genuine $G$-spectra \cite[Theorem~0.1]{Guillou-May}.

Thus under the Guillou--May correspondence, one would anticipate that there is a construction of equivariant algebraic $K$-theory as a spectral Mackey functor. Indeed such a construction is given by Malkiewich and Merling, providing a \textit{Waldhausen spectral Mackey functor} associated to a Waldhausen $G$-category \cite[Proposition~4.11]{MM20}, which is equivalently a spectral Mackey functor \cite[Theorem~2.18]{MM20}. This construction is closely related to that of Barwick \cite{SMF1}, together with Glasman and Shah \cite{BGS}, which provides an equivariant algebraic $K$-theory in the more general setting of Waldhausen $\infty$-categories. In this paper we will follow the setup from \cite{MM20}. The construction of the equivariant algebraic $K$-theory spectral Mackey functor is completely categorical, and in particular its homotopy has not been investigated. 

The intent of this work is to provide an explicit algebraic description of the homotopy Mackey functors of equivariant algebraic $K$-theory for a $G$-ring. If $R$ is a $G$-ring with action map $\theta: G \to \Aut(R)$, we can define a \textit{twisted/skew group ring} $R_\theta[G]$, which is the group ring $R[G]$ equipped with a multiplication which is twisted by the group action. We prove that the algebraic $K$-groups of twisted group rings assemble to form a Mackey functor.

\begin{theorem} (As~\autoref{thm:Mackey-func}) Let $R$ be a $G$-ring, where $G$ is a finite group whose order is invertible over $R$. Then for any $n\ge 0$, there is a Mackey functor $\underline{K}_n^G(R)$, whose value at $G/H$ is the algebraic $K$-theory $K_n(\twR{H})$, and whose restriction and transfer maps are given by $K_n$ applied to extension and restriction of scalar functors over the twisted group rings $\twR{H}$.
\end{theorem}

In particular in the case where the group action is trivial, we recover Mackey functors of algebraic $K$-theory of group rings as built by Dress and Kuku (\autoref{prop:comparison-Kuku}).

As the construction of the Mackey functors $\underline{K}_n^G(R)$ follow by defining transfer and restrictions, applying the $K$-theory spectrum, then taking homotopy groups, we may see that each $\underline{K}_n^G(R)$ is actually the $n$th homotopy of a broader construction, namely a spectral Mackey functor $\mathbf{K}_G(R)$.

\begin{theorem} (As~\autoref{thm:equivalence-Mackey-functors}) Let $G$ be a finite group, and $R$ a $G$-ring where $|G|^{-1}\in R$. Then the Mackey functors $\underline{K}_n^G(R)$ comprise the homotopy of a spectral Mackey functor $\mathbf{K}_G(R)$, which is equivalent to equivariant algebraic $K$-theory of $R$, as defined in \cite{MM}.
\end{theorem}

This theorem allows us to conclude that $\underline{K}_n^G(R)$ is the $n$th homotopy Mackey functor $\underline{\pi}_n$ of equivariant algebraic $K$-theory.

\begin{remark} In their recent book, Balmer and Dell'Ambrogio developed a detailed theory of \textit{Mackey 2-functors}, which are an axiomatization of the idea of Mackey functors ``valued'' in additive categories rather than abelian groups \cite{Balmer}. The work laid out in \autoref{sec:Mackey-functor-tw-group-rings} can be thought of as a new example of a Mackey 2-functor in their framework, namely categories of finitely generated projective modules over twisted group rings. Post-composition with any suitable functor from additive categories to abelian groups allows one to recover a classical Mackey functor out of a Mackey 2-functor; in particular this holds for $K_n$. In \autoref{sec:examples}, we use this general idea to produce broad families of Mackey functors, however our specific functors are chosen to factor through connective spectra.
\end{remark}

In \autoref{sec:preliminaries}, we establish some categorical and algebraic background which will be needed for the paper. We describe the orbit category and the Burnside category for a group $G$, and discuss the categorical and axiomatic definitions of a Mackey functor. We define group actions on categories and their homotopy fixed points, culminating in the computation of the homotopy fixed points of a module category, arising from a group action on a ring, as being the module category of the twisted group ring. Finally, we define equivariant algebraic $K$-theory as a Waldhausen spectral Mackey functor.

In \autoref{sec:Mackey-functor-tw-group-rings} we build various ring homomorphisms between twisted group rings, and use these to construct restriction and extension of scalar functors, which will induce restriction and transfer for our family of Mackey functors $\underline{K}_n^G(R)$. We verify that all the axioms for a Mackey functor hold, modulo some work defining and comparing right actions on certain modules, which is deferred until \autoref{sec:bimodule}. This section proves \autoref{thm:Mackey-func}, that algebraic $K$-groups of twisted group rings form a Mackey functor. We compare these Mackey functors to those constructed by Dress and Kuku in \autoref{subsec:comparison-Kuku}.

In \autoref{sec:comparison}, we compare our restriction, transfer, and conjugation data to that of Malkiewich and Merling, and prove that they agree up to natural isomorphism. In particular this allows us to prove \autoref{thm:equivalence-Mackey-functors}; that the transfer data we defined turns the $K$-theory of twisted group rings into a spectral Mackey functor, which is isomorphic to equivariant algebraic $K$-theory. In \autoref{subsec:mf-from-smf} we explain how to recover the homotopy Mackey functors from a spectral Mackey functor, following work of Bohmann and Osorno. This allows us to conclude that $\underline{K}_n^G(R)$ is the $n$th homotopy Mackey functor of equivariant algebraic $K$-theory.

Finally, in \autoref{sec:examples}, we demonstrate that our work generates various families of Mackey functors. We approach the requirements used to prove \autoref{thm:Mackey-func} from a more axiomatic lens, and generate a Mackey functor on twisted group rings for any suitable invariant in \autoref{subsec:MF-additive-invariants}. As examples, we observe that Hochschild homology, topological Hochschild homology, topological cyclic homology, topological restriction homology, and topological periodic homology all yield Mackey functors in each degree when evaluated on twisted group rings. In \autoref{subsec:MF-Galois}, we recover the classical example that $G/H \mapsto K_n(L^H)$ is a Mackey functor, where $G$ is the Galois group of a Galois extension of fields $L/k$, and we extend this result to hold for Galois ring extensions and for suitable invariants such as $\THH$. In \autoref{subsec:MF-endo-rings}, we use a classical theorem of Auslander to prove that the algebraic $K$-groups of endomorphism rings over fixed subrings $K_n(\End_{R^H}(R))$ form a Mackey functor under certain technical conditions.

\subsection{Acknowledgements}
We would like to thank Mona Merling for guidance throughout this paper, and for inspiring a love of all things equivariant. We would also like to thank Kirsten Wickelgren for being a constant source of mathematical support. Thank you also to Cary Malkiewich for helpful correspondence about homotopy invariants. The author is supported by an NSF Graduate Research Fellowship (DGE-1845298).

\section{Preliminaries}\label{sec:preliminaries}

\subsection{The orbit and Burnside categories}
Denote by $\mathcal{F}_G$ the category whose objects are finite sets equipped with a group action by $G$, and whose morphisms are equivariant functions. It is a general fact that finite $G$-sets decompose into a disjoint union over their orbits, on which $G$ acts transitively. Moreover, any transitive $G$-set can be seen to be of the form $G/H$ for some subgroup $H \subseteq G$. To that end, we define the \textit{orbit category}, denoted $\O_G$, to be the full subcategory of $\mathcal{F}_G$ on the $G$-sets of the form $G/H$. One may easily see that any morphism $G/H \to G/K$ in the orbit category is determined by where it sends the identity coset, and in particular such a morphism exists if and only if $H$ is subconjugate to $K$, meaning that there is some $g\in G$ so that $gHg^{-1} \subseteq K$. As we will be working with conjugation frequently, we incorporate some standard notation.

\begin{notation} Given any subgroup $H \subseteq G$, and $g\in G$, one denotes by ${}^{g}H := gHg^{-1}$ the conjugation of the subgroup by $g$, and $H^g = g^{-1} H g$ conjugation by its inverse.
\end{notation}

There are two natural classes of morphisms in the orbit category one considers. The first is a \textit{projection morphism}: for subgroups $H \subseteq K$, there is a map $G/H \to G/K$ sending $xH$ to $xK$. The second is a \textit{conjugation morphism}: for any $g\in G$ there is a map $G/H \to G/{}^{g}H$ given by sending $xH$ to $(xg^{-1}){}^{g}H$. One may check that these morphisms generate all morphisms in the orbit category, in the sense that any morphism in $\O_G$ can be written as a conjugation morphism followed by a projection morphism, or vice versa.

Given the category $\mathcal{F}_G$ of finite $G$-sets, we can consider its span category $\Span(\mathcal{F}_G)$, where a morphism from $S$ to $T$ is an isomorphism class of a span of the form $S \from U \to T$, and composition of spans is given by pullback.
We remark that hom sets in $\Span(\mathcal{F}_G)$ are equipped with a binary operation; given $S \from U_1 \to T$ and $S \from U_2 \to T$, we can form the span $S \from U_1 \amalg U_2 \to T$ given by taking the disjoint union of $U_1$ and $U_2$. This operation endows each hom-set with an abelian monoid structure, which we may then group complete.

\begin{definition}\label{def:Burnside-cat} The \textit{Burnside category} $B_G$ is the category whose objects are finite $G$-sets, and whose morphisms are given by
\begin{align*}
    \Hom_{B_G}(S,T) := K \left( \Span(S,T), \amalg \right),
\end{align*}
for $S,T\in \mathcal{F}_G$, where $K$ denotes group completion.
\end{definition}

In particular, suppose that $G/H \to G/K$ is a morphism in $\O_G$. Then we can view it as a span in two natural ways, namely as $G/H \xfrom{\id} G/H \to G/K$, or as $G/K \from G/H \xto{\id} G/H$. This defines two inclusion functors $\O_G \hookto B_G$ and $\O_G^\op \hookto B_G$.

The Burnside category is \textit{pre-additive}, meaning that it is enriched in the category of abelian groups. Recall that a group can be thought of as a category with a single object, whose group structure is encoded by composition of morphisms. Analogously, we may think of a ring as a pre-additive category with a single object, as its hom-set will come equipped with composition and addition, with the caveat that composition distributes over addition. Just as groupoids generalize groups, pre-additive categories generalize rings. In the literature, pre-additive categories are sometimes referred to as \textit{ringoids} or as \textit{rings with several objects}.

A module $M$ over a ring $R$ is determined by the data of a ring homomorphism $R \to \End_\Ab(M)$, which is the same as an additive functor from $R$, viewed as a category with one object, to $\Ab$. We can generalize this definition as follows.

\begin{definition}\label{def:module-over-ringoid} Let $\mathcal{A}$ be a pre-additive category. Then we define a \textit{module over} $\mathcal{A}$ to be an additive functor $\mathcal{A} \to \Ab$.
\end{definition}

With this in mind, we can provide a categorical definition of a Mackey functor.

\begin{definition}\label{def:Mackey-functor-as-module} A \textit{Mackey functor} is a module over the Burnside category, in the sense of \autoref{def:module-over-ringoid}.
\end{definition}

As a Mackey functor is an additive functor, it will preserve disjoint unions. Thus we see that a Mackey functor is completely determined on objects by where it sends transitive $G$-sets. Moreover, we may verify that a Mackey functor is completely determined by the composite functors $\O_G \hookto B_G \to \Ab$ and $\O_G^\op \hookto B_G \to \Ab$. In particular, in order to define the data of a Mackey functor, we must determine where it sends objects, projection morphisms in $\O_G$, and conjugation morphisms in $\O_G$ under both composites, and then verify that certain conditions hold. The images of projection morphisms under these inclusions will be called \textit{restriction} and \textit{transfer}, while the image of conjugation morphisms will still be referred to as \textit{conjugation}. With this data in mind, we may state an axiomatic definition of Mackey functors, equivalent to the one above. This definition may be found in \cite[\S2]{Webb}.

\begin{definition}\label{def:Mackey-functor} A \textit{Mackey functor} is a function
\begin{align*}
    \und{M}: \ob \O_G &\to \Ab
\end{align*}
together with morphisms for all $H \subseteq K \subseteq G$ subgroups and $g\in G$, called \textit{transfer}, \textit{restriction}, and \textit{conjugation}, respectively:
\begin{align*}
    I_K^H : \und{M}(G/K) &\to \und{M}(G/H) \\
    R_K^H : \und{M}(G/H) &\to \und{M}(G/K) \\
    c_g: \und{M}(G/H) &\to \und{M}(G/{}^{g}H)
\end{align*}
satisfying the following axioms for all subgroups $J, K, H \subseteq G$ and $h\in H$
\begin{enumerate}[leftmargin=0.7in] 
    \item[\textbf{MF0}] $I_H^H$, $R_H^H$, and $c_h$ are the identity on $\und{M}(G/H)$. 
    \item[\textbf{MF1}] $R_J^K R_K^H = R_J^H$ for all $J \subseteq K \subseteq H$

    \item[\textbf{MF2}] $I_K^H I_J^K = I_J^H$ for all $J \subseteq K \subseteq H$.

    \item[\textbf{MF3}] $c_g c_h = c_{gh}$.
    \item[\textbf{MF4}] $R_{{}^{g}K}^{{}^{g}H} c_g = c_g R_K^H$ for all $K \subseteq H$
    
    \item[\textbf{MF5}] $I_{{}^{g}K}^{{}^{g}H} c_g = c_g I_K^H$ for all $K \subseteq H$.
    \item[\textbf{MF6}] \textit{(Mackey decomposition formula)} For $J,K \subseteq H$ subgroups of $G$, we have that
    \begin{align*}
        R_J^H I_K^H = \sum_{x\in [J\backslash H/K]} I_{J\cap {}^{x}K}^J c_x R_{J^x \cap K}^K.
    \end{align*}
\end{enumerate}
\end{definition}

Mackey functors are rich algebraic objects, which are ubiquitous in algebra. For a more thorough account, we refer the reader to the survey paper of Webb \cite{Webb}. We will revisit this discussion when defining spectral Mackey functors, but we first take a detour to define homotopy fixed points for $G$-categories.

\subsection{$G$-categories and homotopy fixed points}
We define a \textit{$G$-category} to be any functor $G \to \Cat$, that is, it is a 1-category $\mathscr{C}$ equipped with a collection of endofunctors $g: \mathscr{C} \to \mathscr{C}$ for each group element, which we think of as inducing an action on the objects and morphisms of $\mathscr{C}$. We refer to a $G$\textit{-functor} as a natural transformation between functors $G \to \Cat$, equivalently it is a functor of 1-categories which is equivariant in the sense that it preserves the action on objects and morphisms. We say two $G$-categories are $G$\textit{-equivalent} if there is an equivariant functor between them which is also an equivalence of categories.

\begin{example}\label{ex:G-set-discrete-cat} Given any finite $G$-set $X$, we can consider it as a $G$-category in two natural ways. The first is by viewing it as a discrete category, which we also denote by $X$, whose objects are points of $X$ and which only has identity morphisms. Another category we denote by $\und{X}$, and refer to as the \textit{translation category}. This has as objects elements $x\in X$, and a morphism $x\to y$ for any $g\in G$ such that $g\cdot x = y$. Any equivariant function between $G$-sets induces a $G$-functor between the associated discrete categories or the associated translation categories in a natural way.
\end{example}

\begin{definition} Let $\mathcal{E}G$ be the category whose objects are the elements in $G$, and which has exactly one morphism in each hom-set. This comes equipped with an action of $G$, acting as $g\cdot c$ on objects, and on morphisms by $g\cdot \left( c\to c' \right) := gc \to gc'$. We observe that $\mathcal{E}G$ and $\underline{G}$ are $G$-isomorphic, thus we will refer to $\mathcal{E}G$ as the \textit{translation category} of the group $G$ without ambiguity.
\end{definition}

\begin{remark}\label{rmk:EH-EG} For $H \subseteq G$ a subgroup, there is a natural inclusion of translation categories $\mathcal{E}H \to \mathcal{E}G$. This is an equivalence of categories, however constructing an inverse can not be done canonically in general. Writing $H\backslash G = \left\{ Ha_1, \ldots, Ha_n \right\}$, we can define a functor in the other direction $\mathcal{E}G \to \mathcal{E}H$ as follows: any $g\in G$ lies in a unique right coset $Ha_i$, so $g = ha_i$ for some $h$, and we simply define $\mathcal{E}G \to \mathcal{E}H$ to send $g \mapsto h$.
\end{remark}

We can also endow the functor category between two $G$-categories with an action.

\begin{definition}\label{def:G-action-on-functors} For two $G$-functors $\mathscr{C}$ and $\mathscr{D}$, we define an action of $G$ on the objects of $\Fun(\mathscr{C},\mathscr{D})$ by
\begin{align*}
    (gF)(c) := g\cdot \left( F(g^{-1}c) \right).
\end{align*}
Thus the $G$-fixed points $\Fun(\mathscr{C},\mathscr{D})^G$ are precisely the $G$-functors. We define an action on the morphisms (natural transformations) $\eta: F \Rightarrow F'$ by
\begin{align*}
    (g\eta)_c := g\cdot \left( \eta_{g^{-1}c} \right),
\end{align*}
which is a morphism $gF(g^{-1}c) \to gF'(g^{-1}c)$. A $G$\textit{-equivariant natural transformation} is one fixed under this action of $G$.
\end{definition}

\begin{definition}\label{def:htpy-fixed-pts} Given a $G$-category $\mathscr{C}$, and a subgroup $H \subseteq G$, we define the \textit{homotopy fixed point category} as
\begin{align*}
    \mathscr{C}^{hH} := \Fun(\mathcal{E}H, \mathscr{C})^H.
\end{align*}
\end{definition}

\begin{remark} There are two definitions of homotopy fixed points for a $G$-category that one might find in the literature, namely $\mathscr{C}^{hH}$ could mean either $\Fun(\mathcal{E}H, \mathscr{C})^H$ or $\Fun(\mathcal{E}G \times G/H, \mathscr{C})^H$, where $G/H$ is the discrete category on the $G$-set $G/H$ in the sense of \autoref{ex:G-set-discrete-cat}. These two categories are easily seen to be $G$-equivalent, however we will be a bit pedantic about this point, particularly since this equivalence does not admit a canonical inverse.
\end{remark}

\begin{proposition}\label{prop:comparison-htpy-fixed-point-definitions} Let $\mathscr{C}$ be a $G$-category, and $H \subseteq G$ a subgroup. Then there is an equivalence of categories
\begin{align*}
    \Fun(G/H \times \mathcal{E}G, \mathscr{C})^G &\to \Fun(\mathcal{E}H, \mathscr{C})^H,
\end{align*}
given by sending a functor $F$ to $F(eH,-)$, and by sending a natural transformation $\eta: F \Rightarrow \eta'$ to the restricted components:
\begin{align*}
    \eta_{(eH,-)} : F(eH,-) \Rightarrow F'(eH).
\end{align*}
\end{proposition}

While this functor is suitably canonical, if we decided to construct a categorical inverse, we would be forced to confront the same ambiguities to defining an inverse to $\mathcal{E}H \to \mathcal{E}G$ as in \autoref{rmk:EH-EG}. Namely, we must pick explicit coset representatives.

\begin{notation} Let $F: \mathcal{E}H \to \mathscr{C}$ be an $H$-equivariant functor between $G$-categories, and let $\{g_i\}$ be a choice of right coset representatives of the subgroup $H$ in $G$. Then we denote by $\widetilde{F}$ the lift of $F$ to a $G$-equivariant functor $\widetilde{F}: \mathcal{E}G \to \mathscr{C}$ obtained by precomposing with an equivalence $\mathcal{E}G \xto{\sim} \mathcal{E}H$. If $\eta: F \Rightarrow F'$ is an $H$-natural transformation, we denote by $\widetilde{\eta}: \widetilde{F} \Rightarrow \widetilde{F'}$ its lift, defined by $\tilde{\eta}_{hg_i} := \eta_h$.
\end{notation}

\begin{proposition}\label{prop:invs-to-categorical-equivalence-htpy-fixed-pts} In the setting of \autoref{prop:comparison-htpy-fixed-point-definitions}, let $H\backslash G = \left\{ Hg_i\right\}_i$ be a choice of right coset representatives. Then there is a functor exhibiting an equivalence of categories 
\begin{align*}
    \Fun(\mathcal{E}H,\mathscr{C})^H &\to \Fun(G/H\times \mathcal{E}G, \mathscr{C})^G \\
    F &\mapsto \left[ \left( gH, g' \right)\mapsto \left( g\cdot \widetilde{F} \right)(g') \right] = \left[ \left( gH,g' \right) \mapsto g\cdot \widetilde{F}(g^{-1}g') \right],
\end{align*}
which sends $\eta: F \Rightarrow F'$ to the natural transformation whose component at $(gH,g')$ is given by $\left( g\cdot \widetilde{\eta} \right)_{g'} = g\cdot \left( \widetilde{\eta}_{g^{-1}g'} \right)$.
\end{proposition}

\subsection{Homotopy fixed points of module categories and twisted group rings}

As a particular example of the discussion above, we will look at a group action on the category of modules over a ring arising from a group action on a ring. In this setting, the homotopy fixed points admit a nice description as module categories over a \textit{twisted group ring}, defined as follows.

\begin{definition}\label{def:tw-grp-rng} Given a group action $\theta: G \to \Aut_\Ring(R)$, one may define the \textit{twisted group ring} $R_\theta[G]$, whose elements are finite formal sums of the form $\sum_i r_i g_i$ with $r_i \in R$ and $g_i \in G$, and whose multiplication is defined by
\begin{align*}
    \left( \sum_i r_i g_i \right)\cdot \left( \sum_j r'_j g'_j \right) = \sum_{i,j} r_i \theta_{g_i}(r'_j) g_i g'_j.
\end{align*}
\end{definition}

Our goal in this section will be to see that, for a group action of $G$ on $R$, we have an equivalence of categories $\Mod(R)^{hH} \simeq \Mod(R_\theta[H])$ for any subgroup $H \subseteq G$. First we must see how an action of $G$ on $R$ induces an action of $G$ on the module category $\Mod(R)$. 

\begin{definition}\label{def:action-on-module-cat} Suppose that $M$ is an $R$-module, and $\theta: G \to \Aut_\Ring(R)$ is a group action. Then we define $gM$ to be equal to $M$ as an abelian group, but equipped with a twisted action map
\begin{align*}
    R \times M \xto{\theta_g \times \id} R \times M \xto{\times }M.
\end{align*}
Given $f: M \to N$ an $R$-module homomorphism, we denote by $gf: gM \to gN$ the morphism $(gf)(m) := f(m)$, that is, it is the same underlying abelian group homomorphism. We may verify that this data specifies a $G$-action on $\Mod(R)$, and we refer to this as the action of $G$ on $\Mod(R)$ \textit{induced by} $\theta$.
\end{definition}

We now look closer at the homotopy fixed point category $\Mod(R)^{hG}$. Recall that an element of the homotopy fixed point category is a $G$-equivariant functor $f: \mathcal{E}G \to \Mod(R)$. As $f(g) = g\cdot f(e)$, we see that it suffices to determine $f$ on objects by specifying the module $M$ mapped to by the identity object $e\in \mathcal{E}G$. As $\mathcal{E}G$ is a groupoid, $f$ determines $R$-module isomorphisms $f(g) : gM \xto{\sim} M$, where $gM$ is as in \autoref{def:action-on-module-cat}.

As a brief point, we have that $f(g) : gM \xto{\sim} M$ is an isomorphism of $R$-modules, however $gM$ is the same abelian group as $M$ equipped with a different module structure. In particular $f(g)$ can be thought of as an abelian group automorphism of $M$ for any $g\in G$. This automorphism is \textit{not} an $R$-module automorphism, since we have that
\begin{align*}
    f(g)(rm) =  \theta_g(r)m.
\end{align*}
We refer to $f(g)$ instead as a \textit{semilinear} $R$-module automorphism of $M$.

\begin{definition} Let $R$ be a ring, and $\theta \in \Aut_{\Ring}(R)$ a ring automorphism of $R$. Then we define a $\theta$-\textit{semilinear} $R$-module homomorphism $f: M \to N$ to be a function satisfying
\begin{itemize}
    \item $f(m+m') = f(m) + f(m')$ for all $m,m' \in M$
    \item $f(rm) = \theta(r)\cdot f(m)$ for all $r\in R$, and $m\in M$.
\end{itemize}
\end{definition}

Thus we can rephrase our discussion of the homotopy fixed point category $\Mod(R)^{hG}$ to say that an element of the homotopy fixed point category is an $R$-module $M$ equipped with an action of $G$, acting via semilinear $R$-module automorphisms. This is also called a \textit{semilinear group action} of $G$ on $M$.

We now return to our comparison of $\Mod(R)^{hG}$ and $\Mod(R_\theta[G])$. It is a classical fact that modules over a twisted group ring are equivalent to modules equipped with a semilinear $G$-action. We may see this, for example, as an explicit natural bijection as \cite[Corollary~3.6]{Brazelton-semilin}, which follows from an adjunction $R/\Ring \rightleftarrows \Grp/\Aut_\Ring(R)$. This natural bijection will provide us essential surjectivity for the following equivalence of categories.

\begin{proposition}\label{prop:htpy-fixed-points-equivalent-module-cat-over-tw-grp-ring} (c.f. \cite[4.3, 4.8]{Mona}) Let $R$ be a $G$-ring. Then there is an equivalence of categories
\begin{align*}
    \Fun(\mathcal{E}G, \Mod(R))^G \xto{\sim} \Mod(\twR{G}),
\end{align*}
given by sending $F$ to the abelian group $M:= F(e)$ equipped with the $\twR{G}$-module structure
\begin{align*}
    \twR{G} &\to \End_\Ab(M) \\
    rg &\mapsto \left( M \xto{F(e,g)} M \xto{r\cdot-} M \right),
\end{align*}
and sending a natural transformation $\eta: F \Rightarrow F'$ to the $\twR{G}$-module homomorphism
\begin{align*}
    \eta_e : M \to N.
\end{align*}
\end{proposition}
\begin{proof} We must first verify that $\eta_e$ is indeed an $\twR{G}$-module homomorphism as claimed. By assumption it is an $R$-module homomorphism, so in particular it is additive, thus it suffices to check that it preserves the $\twR{G}$-action. Let $M:= F(e)$, and $N:= F'(e)$. Then we must verify that $\eta_e ((rg)\cdot m) = (rg)\cdot \eta_e(m)$ for all $m \in M$. That is, that the composites
\begin{align*}
    M \xto{F(e,g)} M \xto{r\cdot-} M \xto{\eta_e} N \\
    M \xto{\eta_e} N \xto{F'(e,g)} N \xto{r\cdot-} N
\end{align*}
agree. In the second composite, we remark that we can rewrite $\eta_e F'(e,g)$ as $F(e,g) \eta_g$ by the naturality of $\eta$. Moreover, as $\eta$ is $G$-equivariant, we see that $\eta_g = g\cdot \eta_e$, and as $gM = M$ as abelian groups, we see that $\eta_g = \eta_e$ is an equality of abelian group homomorphisms. Thus it suffices to see that $r\eta_e(m) = \eta_e(rm)$, which follows immediately by the component $\eta_e$ being an $R$-module homomorphism.

Checking functoriality is immediate, since for $F \overset{\eta}{\Rightarrow} F' \overset{\epsilon}{\Rightarrow} F^{\prime\prime}$, we clearly have that $(\epsilon\circ \eta)_e = \epsilon_e \circ \eta_e$, and this can be seen to be associative. Moreover, $(\id_F)_e = \id_{F(e)}$, thus the assignment $\Fun(\mathcal{E}G, \Mod(R))^G \to \Mod_{\twR{G}}$ is indeed a functor.

If $\eta,\epsilon : F\Rightarrow F'$ are two $G$-equivariant natural transformations so that $\eta_e = \epsilon_e$ as $\twR{G}$-module homomorphisms, one can see that $\eta_e = \epsilon_e$ agree in particular as $R$-module homomorphisms. As they are equivariant, this implies that $\eta_g = g\eta_e = g\epsilon_e = \epsilon_g$ is an equality of $R$-module homomorphisms, implying that $\eta = \epsilon$. Now if $M \xto{f} N$ is any $\twR{G}$-module homomorphism, then we define functors $F,F' : \mathcal{E}G \to \Mod(R)$ given by $F(g):= gM$ and $F'(g) := gN$. We may then define a natural transformation $\eta: F \Rightarrow F'$ given by $\eta_g := g\cdot f$, which we may check is $G$-equivariant. This implies the functor $\Fun(\mathcal{E}G, \Mod(R))^G \to \Mod_{\twR{G}}$ is fully faithful.

For essential surjectivity, as we have remarked, every module over $R_\theta[G]$ arises as an $R$-module equipped with a semilinear $G$-action \cite{Brazelton-semilin}.
\end{proof}

\begin{corollary}\label{cor:descent-to-proj-modules-subcat} \cite[4.9,~4.10,~4.11]{Mona} Under the conditions of \autoref{prop:htpy-fixed-points-equivalent-module-cat-over-tw-grp-ring}, for any subgroup $H \subseteq G$, there is an equivalence of categories
\begin{align*}
    \Fun(\mathcal{E}H, \Mod(R))^H \xto{\sim} \Mod(\twR{H}).
\end{align*}
Moreover, if $|H|^{-1} \in R$, then this equivalence descends to an equivalence on the category of finitely generated projective modules
\begin{align*}
    \Fun(\mathcal{E}H, \mathcal{P}(R))^H \xto{\sim} \mathcal{P}(\twR{H}).
\end{align*}
\end{corollary}

\subsection{Waldhausen $G$-categories} Suppose we are handed a Waldhausen category $\mathscr{C}$ equipped with an action of $G$, whose $K$-theory we want to study. If we wish to study the $K$-theory in a way that sees the symmetry coming from the $G$-action, then a natural expectation would be that the group action is compatible with the Waldhausen structure in the following sense.

\begin{definition} A \textit{Waldhausen} $G$\textit{-category} is a Waldhausen category $\mathscr{C}$ with an action of $G$ by exact functors.
\end{definition}

We could now attempt to study the $K$-theory of a Waldhausen $G$-category $\mathscr{C}$ by taking the $K$-theory of the associated fixed point categories $\mathscr{C}^H$, obtaining a collection of spectra indexed over subgroups of $G$. This naive approach forces us to come to grips with the fact that ordinary fixed points are too coarse -- in general, it will not be true that $\mathscr{C}^H$ is a Waldhausen category \cite[2.1]{MM}. This failure is rectified by passing to homotopy fixed points, and is our main motivation for studying homotopy fixed points in greater detail.

\begin{theorem} \cite[Theorem~2.15]{MM} Let $\mathscr{C}$ be a $G$-Waldhausen category, and $H \subseteq G$ a subgroup. Then the homotopy fixed points $\mathscr{C}^{hH}$ is a Waldhausen category.
\end{theorem}

This theorem allows us to assign to every orbit $G/H$ a spectrum $K \left( \mathscr{C}^{hH} \right)$, and moreover we will see that we are able to travel between these spectra along morphisms which are directly analogous to those found in Mackey functors. This data forms a homotopical analogue of a Mackey functor, which is referred to as a \textit{spectral Mackey functor}. Thus for any group action on a ring, we will obtain a spectral Mackey functor, which encodes the $K$-theory of the associated homotopy fixed point module categories.

\subsection{Spectral Mackey functors}

One of the primary philosophies of stable homotopy theory is the analogy between abelian groups and spectra. From this perspective, we can try to replicate the definition of a Mackey functor as a module over the pre-additive Burnside category.

\begin{definition} The \textit{spectral Burnside category}, denoted $\mathcal{B}_G$, is defined to be the category whose objects are finite $G$-sets, and whose hom sets are defined by taking the $K$-theory of the permutative category of finite equivariant spans between two finite $G$-sets (see \cite{Guillou-May} for details).
\end{definition}

\begin{definition} A \textit{module} over a spectrally enriched category $\mathcal{A}$ is a spectrally enriched functor $\mathcal{A} \to \Sp$. This is the perspective of Schwede and Shipley, see for example \cite[\S3.3]{SS01}.

\end{definition}

\begin{definition} A \textit{spectral Mackey functor} or $\mathcal{B}_G$\textit{-module} is a spectrally enriched functor $\mathcal{B}_G \to \Sp$.
\end{definition}

The Guillou--May theorem establishes a powerful connection between spectral Mackey functors and genuine $G$-spectra, namely that their homotopy theories are equivalent.

\begin{theorem}\cite[Theorem~1.13]{Guillou-May} When $G$ is finite, there is a string of Quillen equivalences between $\mathcal{B}_G$-modules and genuine $G$-spectra
\begin{align*}
    \Mod(\mathcal{B}_G) \simeq G\text{-}\Sp.
\end{align*}
\end{theorem}

In establishing the theory of equivariant $A$-theory, Malkiewich and Merling construct an alternative version of the spectral Burnside category, which is more understandable for the context of Waldhausen $G$-categories.

\begin{definition} We define $\mathcal{B}_G^\Wald$ to be the spectrally enriched category whose objects are transitive $G$-sets, and whose homs are given by first denoting by $S_{H,K}$ the category of finite $G$-sets containing $G/H \times G/K$ as a retract, and then defining
\begin{align*}
    \Hom_{\mathcal{B}_G^\Wald} \left( G/H, G/K \right) := K \left( S_{H,K} \right),
\end{align*}
where $K$ denotes the Waldhausen $K$-theory spectrum \cite[4.4]{MM}.
\end{definition}

We refer to a module over $\mathcal{B}_G^\Wald$ as a \textit{Waldhausen spectral Mackey functor}. It is natural to ask whether it is substantially different to be a $\mathcal{B}_G^\Wald$-module than to be a $\mathcal{B}_G$-module. The answer is that they are equivalent.

\begin{theorem} \cite[Theorem~2.18]{MM20} There is an equivalence of spectrally enriched categories between $\mathcal{B}_G^\Wald$ and $\mathcal{B}_G$.
\end{theorem}

\begin{corollary}\label{cor:GB-and-GBWald-equiv} \cite[6.1]{SS03} The module categories $\Mod(\mathcal{B}_G)$ and $\Mod(\mathcal{B}_G^\Wald)$ are Quillen equivalent.
\end{corollary}

One of the major results of \cite{MM20} is that, given a $G$-Waldhausen category $\mathscr{C}$, one may endow the assignment $G/H \mapsto K \left( \mathscr{C}^{hH} \right)$ with the structure of a Waldhausen spectral Mackey functor \cite[Proposition~4.11]{MM20}. In particular, the authors define certain transfer, conjugation, and restriction maps at the level of $G$-categories, and prove that these determine the data of a $\mathcal{B}_G^\Wald$-module. By \autoref{cor:GB-and-GBWald-equiv}, this is then a spectral Mackey functor.

\begin{definition} Let $R$ be a $G$-ring. When $\mathscr{C} = \mathcal{P}(R)$ is the category of finitely generated projective $R$-modules, we refer to the spectral Mackey functor $G/H \mapsto K \left( \mathcal{P}(R)^{hH} \right)$ as \textit{equivariant algebraic} $K$\textit{-theory}.
\end{definition}

Our goal will be first to construct Mackey functors associated to the algebraic $K$-theory of twisted group rings. Following this, we will compare the restriction, transfer, and conjugation functors that we define at the level of $G$-categories to those found in \cite{MM20}, which determine the structure of a $\mathcal{B}_G^\Wald$-module. In particular by seeing that they agree up to natural isomorphism, the resulting spectral Mackey functors will be equivalent.

\section{Mackey functors on algebraic $K$-theory of twisted group rings}\label{sec:Mackey-functor-tw-group-rings}

The entirety of this section will be devoted to proving the following theorem.

\begin{theorem}\label{thm:Mackey-func} For any $G$-ring $R$, where $G$ is a finite group whose order is invertible over $R$, and for any $n\ge 0$, there is a Mackey functor $\underline{K}_n^G(R)$, whose value at $G/H$ is the algebraic $K$-group $K_n(\twR{H})$, and whose restriction and transfer maps are given by $K_n$ applied to extension and restriction of scalar functors of finitely generated projective modules over the twisted group rings $\twR{H}$.
\end{theorem}

In order to prove this theorem, we will construct ring homomorphisms between the twisted group rings $\twR{H}$, and use extension and restriction along these ring homomorphisms to build the data of our restriction, transfer, and conjugation morphisms for the Mackey functors $\underline{K}_n^G(R)$. In proving \autoref{thm:Mackey-func}, we avoid making reference to the algebraic $K$-theory functors until the very last step. The advantage of this approach is that the Mackey functor axioms can be proved up to natural isomorphism at the level of module categories. Then by applying $K_n$ for any $n$, the Mackey functor axioms will hold on the nose, due to the following remark, which can be seen as a consequence of additivity.

\begin{remark} If $F,G : \mathscr{C} \to \mathscr{D}$ are naturally isomorphic functors of exact categories, then $K_n(F) = K_n(G)$ as abelian group homomorphisms.
\end{remark}

Before discussing twisted group rings in greater detail, we include the following remark which will allow us to make use of more concise notation in the proofs for this section, without overburdening the reader with extraneous sums and indices.

\begin{remark}\label{rmk:pure-elements} Recall that the elements of a twisted group ring $\twR{G}$ are finite sums of the form $\sum_i r_i g_i$, with $r_i \in R$ and $g_i \in G$. At many points in this paper, it will suffice to prove facts about the group rings $\twR{G}$ on the ``pure elements,'' which are sums over a singleton set, i.e. elements of the form $rg$. For example, in order to define a ring homomorphism $f: \twR{G} \to S$ out of a twisted group ring, it suffices to define $f$ on pure elements, extend this definition additively (i.e. define $f(\sum_i r_i g_i) := \sum_i f(r_i g_i)$), and then verify that $f((r_1 g_1) \cdot (r_2 g_2)) = f(r_1g_1)\cdot f(r_2 g_2)$. Similarly when we know a function out of a twisted group ring to be additive, it suffices to verify that it is multiplicative on pure elements.
\end{remark}

\begin{notation} Given a ring homomorphism $f: R \to S$, we will use $f^\ast : \Mod(S) \to \Mod(R)$ to denote restriction of scalars, and we will use $f_! : \Mod(R) \to \Mod(S)$ to denote extension by scalars. This is intended to align with the notation found in \cite[\S4.4]{MM}, and agrees with the convention from six functors formalism that superscripts denote contravariant assignments, while subscripts are covariant.
\end{notation}

\subsection{Assignments of ring homomorphisms}
As discussed above, we would like to assign ring homomorphisms between group rings to certain classes of maps in the orbit category $\O_G$. Our lives will be easier if this can be done functorially, however ambiguities arising from choosing coset representative in the orbit category will hinder our ability to construct a functor $\O_G \to \Ring$. We can resolve these ambiguities by defining a ``fattened'' version of the orbit category, denoted $\widetilde{\O}_G$.

\begin{remark}\label{rmk:orbit-cat-representatives} Suppose $G/H,G/K$ are two transitive $G$-sets. Then any morphism $f: G/H \to G/K$ in $\O_G$ is determined uniquely by where it sends the identity coset; if $f(eH) = xK$, then because $f$ is $G$-equivariant, we have that $f(gH) = g\cdot f(eH) = gxK$. Thus we can encode the data of $f$ as a morphism in $\O_G$ concisely by the triple $(H, x, K)$. We note however that $(H,x,K) = (H, xk, K)$ for any $k\in K$, so morphisms in $\O_G$ can be represented by many different triples. We define a version of the orbit category in which morphisms are uniquely represented by a choice of coset representative.
\end{remark}

\begin{definition}\label{def:fattened-OG} Define the category $\widetilde{\O}_G$ to be the category whose objects are transitive $G$-sets $G/H$, and whose morphisms are triples $(H, x, K)$ for each element $x\in G$ satisfying $x^{-1}H x \subseteq K$. Composition is given by
\begin{align*}
    (J, y, K) \circ (H, x, J) = (H, xy, K).
\end{align*}
\end{definition}

We contrast this with $\O_G$, where we had that $(H, x, K) = (H, hx, K)$ for any $h\in H$ in this notation. These become distinct morphisms in $\widetilde{\O}_G$.

\begin{proposition} There is a functor
\begin{align*}
    \Tau: \widetilde{\O}_G &\to \Ring \\
    G/H &\mapsto \twR{H} \\
    (H, x, K) &\mapsto \left[ \tau_x: rh \mapsto \phi_{x^{-1}}(r) x^{-1}hx \right].
\end{align*}
\end{proposition}
\begin{proof} We must check that $\tau_x := \Tau(H,x,K)$ is a ring homomorphism. It is additive and preserves multiplicative and additive identities, so it suffices to check it is multiplicative. We see that
\begin{align*}
    \tau_x(r_1 h_1) \cdot \tau_x(r_2 h_2) &= \left( \phi_{x^{-1}}(r_1) x^{-1} h_1 x \right)\cdot \left(  \phi_{x^{-1}}(r_2) x^{-1} h_2 x \right) \\
    &= \phi_{x^{-1}}(r_1) \phi_{x^{-1} h_1 x}(\phi_{x^{-1}}(r_2)) x^{-1} h_1 x x^{-1} h_2 x = \phi_{x^{-1}}(r_1 \phi_{h_1}(r_2)) x^{-1} h_1 h_2 x\\
     &= \tau_x \left( r_1 \phi_{h_1}(r_2) h_1 h_2 \right) =\tau_x((r_1 h_1)\cdot (r_2 h_2)).
\end{align*}
We observe that $\tau_y \circ \tau_x = \tau_{xy}$, and we conclude that $\Tau$ is a functor.
\end{proof}

\begin{proposition} Let $H \subseteq G$ a subgroup, and let $x\in N_G(H)$ be an element of its normalizer. Then $\Tau(H, x, H)$ is a ring automorphism of $\twR{H}$, with inverse $\Tau(H, x^{-1}, H)$.
\end{proposition}

\begin{notation} There are two classes of ring homomorphisms in the image of $\Tau$ that we will interact with frequently, associated to our projection and conjugation morphisms in the orbit category, so we introduce concise notation to distinguish these. For a subgroup $H \subseteq K$, we denote by $\rho_H^K = \Tau(H,e,K)$:
\begin{align*}
    \rho_H^K : \twR{H} &\to \twR{K} \\
    rh &\mapsto rk.
\end{align*}
Given $H \subseteq G$ a subgroup and $g\in G$ arbitrary, we denote by $\gamma^g$ the map $\Tau(H, g^{-1}, {}^{g}H)$:
\begin{align*}
    \gamma^g : \twR{H} &\to \twR{{}^{g}H} \\
    rh &\mapsto \theta_g(r) ghg^{-1}.
\end{align*}
\end{notation}

In our Mackey functors, extension of scalars along $\rho_H^K$ will play the role of transfers, restriction of scalars along $\rho_H^K$ will be our restrictions, and finally extension along $\gamma^g$ will play the role of our conjugation maps.

\subsection{The module categories $\Mod_{\twR{H}}$}

The morphism $\rho_H^K : \twR{H} \to \twR{K}$ exhibits $\twR{K}$ as an $\twR{H}$-module. Really this is an injective inclusion of a subring, which leads us to wonder whether $\twR{K}$ is a free $\twR{H}$-module. This turns out to be true. There are other isomorphic copies of $\twR{H}$ embedded in $\twR{K}$, and we can translate between them by right multiplication by elements of $K$.

\begin{remark}\label{rmk:shift-map} Given any twisted group ring $\twR{K}$, and $y\in K$, we have a \textit{shift map}
\begin{align*}
    \sh_y : \twR{K} &\to \twR{K} \\
    rk &\mapsto rky.
\end{align*}
This is not a ring automorphism of $\twR{K}$, but it is an $\twR{H}$-module automorphism for any $H \subseteq K$.
\end{remark}
\begin{proof} Let $\twR{K} \in \Mod_{\twR{H}}$ under the structure map $\rho_H^K$ for any subgroup $H \subseteq K$, let $rh \in \twR{H}$ and $r' k$ be a pure element of the ring $\twR{K}$. Then we see that
\begin{align*}
    \sh_y \left( (rh)\cdot \left( r' k \right) \right) &= \sh_y \left( r \phi_h(r') hk \right) = r\phi_h(r') hky \\
    &= (rh)\cdot (r' k y) = (rh)\cdot \sh_y(r' k).
\end{align*}
It is clear that $\sh_y$ is bijective on pure elements, and therefore as an endomorphism of $\twR{K}$ as an $\twR{H}$-module.
\end{proof}

Provided these shift maps, we can describe the structure of $\twR{K}$ as an $\twR{H}$-module via the following result.

\begin{proposition}\label{prop:pi-gives-free-module} Suppose $G$ is finite, and let $H \subseteq K$ be subgroups of $G$. Then $\twR{K}$ is a free $\twR{H}$-module under the morphism $\rho_H^K$.
\end{proposition}
\begin{proof} We remark that $\rho_H^K$ is injective as an $\twR{H}$-module homomorphism, and that the shift maps were injective as in \autoref{rmk:shift-map}, therefore $\sh_y \circ \rho_H^K: \twR{H} \to \twR{K}$ is injective. Picking right coset representatives $y_1, \ldots, y_n$ for $H\backslash K$, we claim that the ring homomorphism
\begin{align*}
    (\sh_{y_1}\circ \rho_H^K, \ldots, \sh_{y_n}\circ \rho_H^K): \bigoplus_{y_i \in H \backslash K} \twR{H} \to \twR{K}
\end{align*}
is an $\twR{H}$-module isomorphism. As this is a direct sum of injective module homomorphisms, and their images can be checked to intersect trivially in $\twR{K}$, thus the map is clearly injective. Moreover, it is easy to see it is surjective, as every element $rk \in \twR{K}$ can be written as $r h y_i$ for some $h\in H$ and $i$, and thus is hit by the homomorphism $\sh_{y_i} \rho_H^K$.
\end{proof}

\begin{corollary}\label{cor:right-coset-reps-give-basis} Let $H \subseteq K$ subgroups of $G$, and let $y_1, \ldots, y_n$ be right coset representatives for $H\backslash K$. Then the set
\begin{align*}
    \left\{ 1_R y_1, \ldots, 1_R y_n \right\}
\end{align*}
is an $\twR{H}$-basis for $\twR{K}$, viewed as a left $\twR{H}$-module.
\end{corollary}

We will also want to understand how bases of free modules are affected under extension along the morphisms $\gamma^x : \twR{H} \to \twR{{}^{x}H}$. We remark that $\gamma^x$ is the restriction of a ring automorphism $\twR{G} \to \twR{G}$ of the form $rg \mapsto \phi_x(r) xgx^{-1}$. By keeping this fact in mind, we are able to better characterize extension of scalars along $\gamma^x$ for certain types of free $\twR{H}$-modules, as is shown by the following remark.

\begin{remark} Let $R,S \subseteq T$ be subrings of a general ambient ring, and let $f: T \to T$ be a ring isomorphism which restricts to an isomorphism of subrings $g:= \left. f \right|_{ R } : R \xto{\sim} S$. Suppose that $M$ is a free $R$-module of the form $\oplus_i R t_i$, where $t_i \in T$. Then the $S$-module $g_! M$ obtained by extension of scalars along $g$ is the free $S$-module of the form $\oplus_i S f(t_i)$.
\end{remark}

\begin{example}\label{ex:restriction-along-conj-basis} As a particular case of this remark, suppose that we have a free $\twR{H}$-module $M$ of the form $M = \oplus \twR{H} 1_R y_i$ for some $y_i \in G$. Then we have that
\begin{align*}
    \gamma^x_! M = \oplus \twR{{}^{x}H} \gamma^x(1_R y_i) = \oplus \twR{{}^{x}H} x y_i x^{-1}.
\end{align*}
\end{example}

In \autoref{subsec:MF6} we will need to prove complicated natural isomorphisms related to extension and restriction of scalars along these homomorphisms of twisted group rings, and \autoref{ex:restriction-along-conj-basis} will come into play.

\subsection{Restriction and transfer}\label{subsec:res-transfer-Mn} In this section we will define extension and restriction of scalars along homomorphisms of twisted group rings and verify their basic properties.

\begin{proposition}\label{prop:res-tr-functors} For subgroups $H \subseteq K$ of $G$, the morphism $\rho_H^K: R_\theta[H] \to R_\theta[K] $ induces extension and restriction of scalar functors, which restrict to the subcategory of finitely generated projective modules, and induce functors which we denote by
\begin{align*}
    \Tr_K^H := \left( \rho_H^K \right)_!:\mathcal{P}(\twR{H}) &\to \mathcal{P}(\twR{K}) \\
    \Res_K^H := \left( \rho_H^K \right)^\ast : \mathcal{P}(\twR{K}) &\to \mathcal{P}(\twR{H}).
\end{align*}
\end{proposition}
\begin{proof} We recall that extension of scalars always preserves projective modules. Restriction of scalars does not in general preserve projective modules, however as $R_\theta[K]$ is finitely generated and projective over $R_\theta[H]$ by \autoref{prop:pi-gives-free-module}, restriction of scalars along $\rho_H^K$ descends to the subcategories of finitely generated projective modules.
\end{proof}

We denote by $\gamma^g_!: \mathcal{P}(\twR{H}) \to \mathcal{P}(\twR{{}^{g}H})$ the extension of scalars functor along the ring homomorphism $\gamma^g: \twR{H} \to \twR{{}^{g}H}$. 

\begin{proposition}\label{prop:our-functors-are-exact} Let $G$ be finite. Then $\Tr_K^H$, $\Res_K^H$, and $\gamma^g_!$ are exact.
\end{proposition}
\begin{proof} By \autoref{prop:pi-gives-free-module}, we have that $R_\theta[K]$ is free over $R_\theta[H]$, and in particular it is flat. Therefore extension of scalars (which is tensoring $R_\theta[H]$-modules with $R_\theta[K]$) is exact. Extension along $\gamma^g$ is exact as $\gamma^g$ is a ring isomorphism. Finally, restriction of scalars is always exact since any exact sequence of $R_\theta[K]$-modules will remain exact (as the underlying sets and functions are unchanging) when viewed as an $R_\theta[H]$-module.
\end{proof}

\begin{lemma}\label{lem:webb-mackey-functor-definition} Let $J \subseteq K \subseteq H \subseteq G$ be subgroups, and let $g\in G$. Then the following hold (where the reader is encouraged to compare the indexing here with that of \autoref{def:Mackey-functor}):
\begin{enumerate}[(1)]
    \setcounter{enumi}{-1}
    \item[\textbf{(0)}] $\Tr_H^H$ and $\Res_H^H$ are equal to the identity functor on $\mathcal{P}(\twR{H})$, and $\gamma^h_!$ is naturally isomorphic to the identity for all $h\in H$.

    \item[\textbf{(1)}] We have a natural isomorphism
    \begin{align*}
        \Res_J^K \Res_K^H \cong \Res_J^H.
    \end{align*}

    \item[\textbf{(2)}] We have a natural isomorphism

    \begin{align*}
        \Tr_K^H \Tr_J^K \cong \Tr_J^H.
    \end{align*}

    \item[\textbf{(3)}] We have a natural isomorphism \[\gamma^g_! \gamma^h_! \cong \gamma^{gh}_!\]
    
    \item[\textbf{(5)}] We have a natural isomorphism

    \begin{align*}
        \Tr_{{}^{g}K}^{{}^{g}H} \gamma^g_! &\cong \gamma^g_! \Tr_K^H.
    \end{align*}
\end{enumerate}
\end{lemma}
\begin{proof} For (0), it is clear that $\Tr_H^H$ and $\Res_H^H$ are the identity. To see that $\gamma^h_!$ is naturally isomorphic to the identity on $\mathcal{P}(\twR{H})$, it suffices to see that $\gamma^h$ is a ring automorphism of $\twR{H}$. This is clear, since $\gamma^h \circ \gamma^{h^{-1}} = \gamma^e = \id$.

The statements (1), (2), (3), and (5) are basic consequences of the functoriality of restriction and extension of scalars.

\end{proof}

\begin{proposition} Let $H \subseteq G$ be a subgroup and $g\in G$. Then we have the following.
\begin{enumerate}
        \item[\textbf{(4)}] There is a natural isomorphism of functors
    \begin{align*}
        \Res_{{}^{g}K}^{{}^{g}H} \gamma^g_! &\cong \gamma^g_! \Res_K^H.
    \end{align*}
\end{enumerate}
\end{proposition}
\begin{proof} Consider the commutative square
\[ \begin{tikzcd}
    \twR{H}\rar[hook,"\rho_H^K" above]\dar["\gamma^g" left] & \twR{K}\dar["\gamma^g" right]\\
    \twR{{}^{g}H}\rar[hook, "\rho_{{}^{g}H}^{{}^{g}K}" below] & \twR{{}^{g}K}.
\end{tikzcd} \]
Since $\gamma^g$ is an isomorphism of rings, extension along it is an equivalence of categories, and therefore the unit $\id \to (\gamma^g)^\ast \gamma^g_!$ and counit $\gamma^g_! (\gamma^g)^\ast \to \id$ are natural isomorphisms. Thus we see that
\begin{align*}
    \left( \rho_{{}^{g}H}^{{}^{g}K} \right)^\ast \gamma^g_! &\cong \gamma^g_! \left( \gamma^g \right)^\ast \left( \rho_{{}^{g}H}^{{}^{g}K} \right)^\ast \gamma^g_! 
    \cong \gamma^g_! \left( \rho_H^K \right)^\ast (\gamma^g)^\ast \gamma^g_!
    \cong \gamma^g_! \left( \rho_H^K \right)^\ast.
\end{align*}
\end{proof}

Thus we have established natural isomorphisms of restriction, transfer, and conjugation functors of exact categories. By the additivity theorem, after applying $K_n$, we will obtain equalities of abelian group homomorphisms, satisfying the axioms of a Mackey functor. Before we are able to do this, it remains to prove the most difficult of all the axioms for a Mackey functor, which is the Mackey decomposition formula.
\begin{enumerate}[(1)]
    \setcounter{enumi}{5}
    \item[\textbf{(6)}] For $J,K \subseteq H$ subgroups of $G$, we have a natural isomorphism of functors
    \begin{align*}
        \Res_J^H \Tr_K^H \cong \sum_{x\in [J\backslash H/K]} \Tr_{J\cap {}^{x}K}^J \gamma^x_! \Res_{J^x \cap K}^K.
    \end{align*}
\end{enumerate}

Verifying this last axiom will take up the entirety of \autoref{subsec:MF6}, with some of the work deferred to \autoref{sec:bimodule}.

\subsection{The Mackey decomposition formula for twisted group rings}\label{subsec:MF6}

In order to prove that $\Res_J^H \Tr_K^H$ and $\sum_{x\in [J\backslash H/K]} \Tr_{J\cap {}^{x}K}^J \gamma^x_! \Res_{J^x \cap K}^K$ are naturally isomorphic as functors from $\mathcal{P}(\twR{K})$ to $\mathcal{P}(\twR{J})$, we will first verify that their images agree on $R_\theta[K]$. Not only does one obtain isomorphic left $R_\theta[J]$-modules when plugging $R_\theta[K]$ into each of these functors, there is actually an induced right $R_\theta[K]$-module structure for which their images on $R_\theta[K]$ agree as bimodules. By a theorem of Eilenberg and Watts, this will be sufficient to demonstrate that the functors are naturally isomorphic.

\begin{proposition}\label{prop:cardinality-counting-double-cosets} Let $J,K \subseteq H$ be subgroups, and let $x_1, \ldots, x_n$ be a set of double coset representatives for $J \backslash H/K$. Then
\begin{align*}
    |H:J| = \sum_{i=1}^n \left| K : J^{x_i} \cap K \right|.
\end{align*}
\end{proposition}
\begin{proof} From the pullback square of finite $H$-sets
\[ \begin{tikzcd}
    \coprod_{x\in [J\backslash H/K]} H/(J^x \cap K)\rar\dar\pb & H/J\dar\\
    H/K\rar & H/H,\\
\end{tikzcd} \]
we obtain an isomorphism of finite $H$-sets:
\begin{align*}
    H/J \times H/K &\cong \amalg_{x\in [J\backslash H/K]} H/(J^x \cap K).
\end{align*}
As the underlying sets are bijective, we see that
\begin{align*}
    |H/J|\cdot |H/K| &= \sum_{x\in [J\backslash H/K]} |H/(J^x \cap K)| \\
    &= \sum_{x\in [J\backslash H/K]} |H/K|\cdot |K/(J^x \cap K)|.
\end{align*}
Canceling $|H/K|$ from each side gives the desired equality.
\end{proof}

\begin{proposition}\label{prop:RH-basis} For each $x_i\in [J\backslash H/K]$, let $\beta_{i,1}, \ldots, \beta_{i,r_i}$ denote a set of right coset representatives for $(J^{x_i} \cap K)\backslash K$. Then the right cosets for $J$ in $H$ are given by
\begin{align*}
    H = \bigcup_{i=1}^n \bigcup_{\ell=1}^{r_i} J x_i \beta_{i,\ell}.
\end{align*}
\end{proposition}
\begin{proof} By \autoref{prop:cardinality-counting-double-cosets}, the set $\left\{ x_i \beta_{i,\ell} \right\}_{i,\ell}$ gives us the expected number of coset representatives. It thus suffices to check that they are genuine representatives.

Let $h\in H$ be arbitrary. Then there is a unique $x_i$ so that $h$ lies in the double coset $J x_i K$, that is, $h = j x_i k$ for some unique $i$. We may take the element $k$ and see which right coset of $(J^{x_i} \cap K)\backslash K$ it lies in, so there is some unique $\beta_{i,\ell}$ so that $k \in (J^{x_i}\cap K)\beta_{i,\ell}$. Thus there is some unique $j'\in J$ for which $k = x_i^{-1} j' x_i \beta_{i,\ell}$, from which we observe
\begin{align*}
    h = j x_i k &= (jj') x_i \beta_{i,\ell}.
\end{align*}
\end{proof}

\begin{corollary}\label{cor:RH-basis} There is an isomorphism of free $\twR{J}$-modules 
\begin{align*}
    \twR{H} \cong \bigoplus_{i=1}^n \bigoplus_{\ell=1}^{r_i} \twR{J} x_i \beta_{i,\ell}.
\end{align*}
\end{corollary}
\begin{proof} As we have seen in \autoref{cor:right-coset-reps-give-basis}, right coset representatives give bases for twisted group rings as free modules. Combining this with \autoref{prop:RH-basis} gives the desired result.
\end{proof}

\begin{corollary}\label{cor:isom-base-rings} There is an isomorphism of left $\twR{J}$-modules
\begin{align*}
    \Res_J^H \Tr_K^H\left( \twR{K} \right) \cong \bigoplus_{x\in [J\backslash H/K]} \Tr_{J\cap {}^{x} K}^J \gamma^x_! \Res_{J^x \cap K}^K \left( \twR{K} \right).
\end{align*}
\end{corollary}
\begin{proof} We first see that $\Tr_K^H(\twR{K}) = \twR{H}$. Invoking \autoref{cor:RH-basis}, we see that restriction gives us
\begin{align*}
    \Res_J^H \twR{K} \cong \bigoplus_{i=1}^n \bigoplus_{\ell=1}^{r_i} \twR{J} x_i \beta_{i,\ell}.
\end{align*}

On the other hand, for a fixed $x_i$, we have
\[
    \Res_{J^{x_i}\cap K}^K \twR{K} \cong \bigoplus_{\ell=1}^{r_i} \twR{J^{x_i}\cap K} \beta_{i,\ell}.
\]
Via extension along $\gamma^{x_i}$, our basis elements are conjugated (see \autoref{ex:restriction-along-conj-basis}), so we see that
\begin{align*}
    \gamma^{x_i}_! \left(\oplus_j \twR{J^{x_i}\cap K} \beta_{i,j} \right) \cong \oplus_j \twR{J\cap {}^{x_i} K} x_i \beta_{i,j} x_i^{-1}.
\end{align*}
Finally extending along the inclusion $J\cap {}^{x_i} K \subseteq K$, we have that
\begin{align*}
    \Tr_{J\cap {}^{x_i} K}^J \left( \oplus_j \twR{J\cap {}^{x_i} K} x_i \beta_{i,j} x_i^{-1} \right) \cong \oplus_j \twR{J} x_i \beta_{i,j} x_i^{-1}.
\end{align*}
Taking a direct sum of all such paths, we see that
\begin{align*}
    \bigoplus_{i=1}^n \Tr_{J\cap {}^{x_i} K}^J \gamma^{x_i}_! \Res_{J^{x_i} \cap K}^K \left( \twR{K} \right) \cong \bigoplus_{i=1}^n \bigoplus_{\ell=1}^{r_i} \twR{J} x_i \beta_{i,j} x_i^{-1}.
\end{align*}
One sees that the shift map $s_{x_i} : \twR{J}x_i \beta_{i,j}x_i^{-1} \to \twR{J} x_i \beta_{i,j}$ is an isomorphism of free $\twR{J}$-modules, and therefore a direct sum of shift maps exhibits an isomorphism
\begin{align*}
    \bigoplus_{i=1}^n \left( \oplus_j s_{x_i} \right) :  \bigoplus_{i=1}^n \Tr_{J\cap {}^{x_i} K}^J \gamma^{x_i}_! \Res_{J^{x_i} \cap K}^K \left( \twR{K} \right) \xto{\sim} \Tr_K^H \Res_J^H \left( \twR{K} \right).
\end{align*}
\end{proof}

\begin{proposition}\label{prop:bimodule-homomorphism} We have that the left $\twR{J}$-modules $\Res_J^H \Tr_K^H\left( \twR{K} \right)$ and \\$\bigoplus_{x\in [J\backslash H/K]} \Tr_{J\cap {}^{x} K}^J \gamma^x_! \Res_{J^x \cap K}^K \left( \twR{K} \right)$ inherit a compatible right $\twR{K}$-module structure, which allows us to view them as $(\twR{J}, \twR{K})$-bimodules. Moreover, the shift map exhibiting a left $\twR{J}$-module isomorphism in \autoref{cor:isom-base-rings} is a right $\twR{K}$-module homomorphism as well. Therefore they are isomorphic as bimodules.
\end{proposition}
\begin{proof} Deferred to \autoref{sec:bimodule}.
\end{proof}

\begin{lemma}\label{lem:Watts} Let $P,Q: \mathcal{P}(R) \to \mathcal{P}(S)$ be two exact functors from the category of projective left $R$-modules to the category of projective left $S$-modules. If $P(R) \cong Q(R)$ as $(S,R)$-bimodules, then $P$ and $Q$ are naturally isomorphic.
\end{lemma}
\begin{proof} This is a direct corollary of the Eilenberg--Watts theorem \cite{Eilenberg,Watts}, which states that a right exact coproduct-preserving functor between module categories is naturally isomorphic to tensoring with a bimodule. We first remark that the proof of Eilenberg--Watts relies only upon our capacity to obtain free resolutions, so we may restrict our attention to subcategories of projective modules without issue, provided that $P(R)$ and $Q(R)$ are projective in $\Mod(S)$. For any such functors $P$ and $Q$, we have that $P(R)$ and $Q(R)$ obtain natural right $R$-module structures --- this is discussed further in \autoref{sec:bimodule}.

Applying the theorem of Eilenberg and Watts, we have that $P \cong M \otimes_R -$ for some $(S,R)$-bimodule $M$, while $Q \cong N \otimes_R -$, for some $(S,R)$-bimodule $N$. The assumption that there is a bimodule isomorphism $P(R) \cong Q(R)$ implies that $M$ and $N$ are isomorphic as $(S,R)$-bimodules. In particular, there is a natural isomorphism of functors $M \otimes_R - \cong N \otimes_R -$. Combining all these facts, we have a string of natural isomorphisms
\begin{align*}
    P \cong M \otimes_R - \cong N \otimes_R - \cong Q.
\end{align*}
\end{proof}

\begin{corollary}\label{cor:MF6} There is a natural isomorphism of functors
\begin{align*}
    \Res_J^H \Tr_K^H \cong \bigoplus_{x\in [J\backslash H/K]} \Tr_{J\cap {}^{x}K}^J \gamma^x_! \Res_{J^x \cap K}^K
\end{align*}
between the categories of projective modules $\mathcal{P}(\twR{H})$ and $\mathcal{P}(\twR{J})$.
\end{corollary}
\begin{proof} Apply \autoref{lem:Watts} to the following situation: let $R = \twR{K}$, let $S = \twR{J}$, let $P = \bigoplus_{x\in [J\backslash H/K]} \Tr_{J\cap {}^{x} K}^J \gamma^x_! \Res_{J^x \cap K}^K$, and let $Q = \Res_J^H \Tr_K^H$. Then the assumption that $P(R) \cong Q(R)$ as left $S$-modules is given by \autoref{cor:isom-base-rings}, while their isomorphism as bimodules follows from \autoref{prop:bimodule-homomorphism}. The fact that we can apply the theorem in this setting follows from the fact that all functors considered are exact (\autoref{prop:our-functors-are-exact}), and a direct sum of exact functors is exact as well.
\end{proof}

We are now equipped to prove the main theorem of this section.

\subsection{Proof of \autoref{thm:Mackey-func}}

As stated earlier, we have proven compatibility akin to the axioms for a Mackey functor at the level of exact functors between categories of projective modules, and only up to natural isomorphism. Our goal is then to apply an algebraic $K$-group functor to obtain a genuine Mackey functor. We should first verify that this is something we are allowed to do; i.e. will applying $K_n$ to extension and restriction of scalar functors induce algebraic $K$-theory homomorphisms? It turns out that the conditions we used to ensure that projective modules were preserved suffice to induce abelian group homomorphisms on $K$-groups.

\begin{proposition}\label{prop:induced-K-group-maps} If $f: R \to S$ is a ring homomorphism exhibiting $S$ as a finitely generated projective $R$-module, then extension and restriction of scalars induce group homomorphisms
\begin{align*}
    K_n(f_!) : K_n(R) &\to K_n(S) \\
    K_n(f^\ast) : K_n(S) &\to K_n(R),
\end{align*}
for all $n\ge 0$. For a reference, see for example \cite[IV.6.3.2]{Weibel}.
\end{proposition}

Thus for any $n\ge 0$, we have abelian group homomorphisms
\begin{align*}
    K_n(\Tr_K^H) : K_n(\twR{H}) &\to K_n(\twR{K}) \\
    K_n(\Res_K^H) : K_n(\twR{K}) &\to K_n(\twR{H}) \\
    K_n(\gamma^g_!) : K_n(\twR{H}) &\to K_n \left( \twR{{}^{g}H} \right).
\end{align*}

\begin{proof}[Proof of \autoref{thm:Mackey-func}] Let $n\ge 0$ be any integer. Consider the equalities and natural isomorphisms listed in \textbf{(0)}---\textbf{(6)} proved in \autoref{lem:webb-mackey-functor-definition} and \autoref{cor:MF6}. By letting
\begin{align*}
    I_K^H :&= K_n (\Tr_K^H) \\
    R_K^H :&= K_n(\Res_K^H) \\
    c_g :&= K_n(\gamma^g_!),
\end{align*}
one sees by the additivity theorem that these properties become equalities of abelian group homomorphisms, and in particular satisfy the Mackey functor axioms listed in \autoref{def:Mackey-functor}. Thus there is a Mackey functor $\underline{K}_n^G(R)$ whose value on $G/H$ is the algebraic $K$-group $K_n(\twR{H})$, and whose restriction and transfer maps arise from restriction and extension of scalars along ring homomorphisms between twisted group rings.
\end{proof}

\subsection{Comparison with Dress and Kuku's Mackey functor for group rings}\label{subsec:comparison-Kuku}

In \cite{DK-cartan,DK-convenient}, Dress and Kuku proved the existence of a family of Mackey functors defined by taking functors from translation categories of $G$-sets into an exact category $\mathscr{C}$, and then taking their algebraic $K$-theory. In particular when the category $\mathscr{C}$ is a category of projective modules over a ring $R$, this has the interpretation of recovering the algebraic $K$-theory of the group ring $R[G]$. We will provide a short argument that, in the case where $R$ is a ring with trivial action from $G$, our Mackey functor $\underline{K}_n^G(R)$ constructed above agrees with the Mackey functor $K_n(R[G])$ defined by Dress and Kuku. For further detail on the Mackey functor $K_n(R[G])$ and the more general construction behind it, we refer the reader to the excellent exposition found in \cite[Chapter~10]{Kuku-rep-thy}.

If $\mathscr{C}$ is an exact category and $X$ is any $G$-set, then the functor category $\Fun(\und{X}, \mathscr{C})$ is exact, where we understand exactness of natural transformations to mean exactness pointwise \cite[Theorem~10.1.1]{Kuku-rep-thy}. Then by applying $K_n$ for any $n$, we obtain an abelian group $K_n\Fun(\und{X}, \mathscr{C})$. This is denoted by $K_n^G(X, \mathscr{C})$ in the relevant literature.

\begin{theorem}\cite[Theorem~10.1.2]{Kuku-rep-thy} We have that
\begin{align*}
    K_n^G(-, \mathscr{C}) : G\Set &\to \Ab
\end{align*}
is a Mackey functor for all $n\ge 0$.
\end{theorem}

The description of restriction and transfer maps can be found in detail in \cite{DK-convenient} or \cite{Kuku-rep-thy}. We will be interested, however, in the particular case where $\mathscr{C} = \mathcal{P}(R)$ is the category of projective left $R$-modules. In this context, we have that $\Fun(\und{G/H}, \mathcal{P}(R))$ is equivalent to the category of finitely generated projective left $R[H]$-modules \cite[Theorem~3.2]{DK-convenient}. In particular, the restriction and transfer maps can be understood as restriction and extension of scalars for the associated group rings \cite[Remarks~10.3.1(2)]{Kuku-rep-thy}. From these facts, we conclude the following.

\begin{proposition}\label{prop:comparison-Kuku} Let $G$ be a finite group, and let $R$ be a ring which is viewed as having a trivial $G$-action. Then there is an isomorphism of Mackey functors for all $n\ge 0$
\begin{align*}
    K_n^G(R[G]) \cong \underline{K}_n^G(R),
\end{align*}
where $K_n^G(R[G])$ is the Mackey functor constructed in \cite[\S3]{DK-convenient}, and $\underline{K}_n^G(R)$ is the Mackey functor from \autoref{thm:Mackey-func}.
\end{proposition}

\section{Comparison with equivariant algebraic $K$-theory}\label{sec:comparison}
Given a $G$-Waldhausen category $\mathscr{C}$, one may give the assignment $G/H \mapsto K \left( \mathscr{C}^{hH} \right)$ the structure of a Waldhausen spectral Mackey functor, that is, the structure of a $\mathcal{B}_G^\Wald$-module. We will refer to such an object as a \textit{homotopy fixed points Waldhausen spectral Mackey functor}. Malkiewich and Merling constructed restriction and transfer functors for such objects prior to taking $K$-theory, and then proved that their restriction and transfer maps satisfied the necessary axioms required to specify a spectral Mackey functor \cite[4.11]{MM}.

Suppose that $\mathscr{C} = \mathcal{P}(R)$ is the category of finitely generated projective modules over a ring $R$, where $G$ acts on $\mathcal{P}(R)$ via an action of $G$ on the ring $R$. Then we have shown that one can define transfer and restriction functors on the categories of homotopy fixed points as extension and restriction of scalars along homomorphisms between twisted group rings. Let $\mathbf{K}_G(R)$ denote the data of the assignment $G/H \mapsto K(R_\theta[H])$, equipped with restriction, transfer, and conjugation functors as defined in \autoref{subsec:res-transfer-Mn}. If we can prove that our transfers and restrictions agree up to natural isomorphism with those constructed by Malkiewich and Merling, then we will have proven that $\mathbf{K}_G(R)$ is a Waldhausen spectral Mackey functor which is equivalent to the homotopy fixed points Waldhausen spectral Mackey functor for $\mathcal{P}(R)$.

\begin{theorem}\label{thm:equivalence-Mackey-functors} Suppose that $G$ is a finite group with an action on $R$, and $|G|^{-1}\in R$. Then we have that $\mathbf{K}_G(R)$ is a $\mathcal{B}_G$-module, and is equivalent to the $\mathcal{B}_G$-module $G/H \mapsto K(\mathcal{P}(R))^{hH}$ as defined by Malkiewich and Merling.
\end{theorem}

As a corollary, we can compute the homotopy Mackey functors of equivariant algebraic $K$-theory in terms of the Mackey functors $\underline{K}_n^G(R)$ constructed in \autoref{thm:Mackey-func}, whose transfers and restrictions admit a tractable algebraic description.

\begin{corollary}\label{cor:KGn-computes-pi-n-KGR} We have that $\underline{K}_n^G(R)$ is the $n$th homotopy Mackey functor of the equivariant algebraic $K$-theory spectral Mackey functor.
\end{corollary}

\subsection{The homotopy fixed points Waldhausen spectral Mackey functor}

For this section, let $\mathscr{C}$ be a $G$-category with arbitrary coproducts.

\begin{definition}\label{def:MM-restriction}\cite[4.7]{MM} Suppose $f: S \to T$ is a map of finite $G$-sets. Then we have a \textit{restriction functor}
\begin{align*}
    f^\ast : \Fun(T \times \mathcal{E}G, \Mod(R))^G &\to \Fun(S \times \mathcal{E}G, \Mod(R))^G \\
    F &\mapsto \left(f^\ast F : (s,g) \mapsto F(f(s),g) \right) \\
    f^\ast F \left( s, g\to g' \right) &= F(f(s), g\to g'),
\end{align*}
and on $\alpha: F \Rightarrow F'$ by the formula
\begin{align*}
    (f^\ast F)(s,g) = F(f(s),g) \xto{\alpha} F'(f(s),g) = (f^\ast F')(s,g).
\end{align*}
\end{definition}

\begin{definition}\label{def:MM-transfer}\cite[4.7]{MM} Let $f: S \to T$ be a map of finite $G$-sets. Then we have a \textit{transfer functor}
\begin{align*}
    f_! : \Fun(S\times \mathcal{E}G, \mathscr{C})^G \to \Fun(T\times \mathcal{E}G, \mathscr{C})^G
\end{align*}
is given by sending $F \mapsto f_! F$, where
\begin{align*}
    (f_! F)(t,g) = g \left( \bigoplus_{i\in f^{-1} \left( g^{-1}t \right)} F(i,e) \right).
\end{align*}
\end{definition}

In particular when $S \to T$ is a map between transitive $G$-sets, these definitions provide restriction and transfer maps between homotopy fixed point categories.

\begin{proposition} \cite[4.11]{MM} The restriction and transfer functors defined above for transitive $G$-sets turn $K(\mathscr{C}^{hH})$ into a spectral Mackey functor.
\end{proposition}

In the setting where $\mathscr{C} = \mathcal{P}(R)$, our goal will be to compare the restriction and transfer maps defined above with those given by restriction and extension of scalars along ring homomorphism between twisted group rings. As we may see, it suffices to define restriction and transfers on the orbit category, and therefore it suffices to consider \autoref{def:MM-restriction} and \autoref{def:MM-transfer} in the cases where we have a projection map $G/H \to G/K$ for a subgroup $H \subseteq K$, or a conjugation map $G/H \to G/{}^{g}H$.

Our goal will be to first compare restriction along a projection $G/H \to G/K$ in the sense of \autoref{def:MM-restriction} with restriction of scalars along $\rho_H^K$. We will also compare restriction along a conjugation function $c_x: G/H \to G/{}^{x}H$ with restriction of scalars along $\gamma^x: R_\theta[H] \to R_\theta[{}^{x}H]$. As every morphism in the orbit category is a composite of projection and conjugation, this will imply that our definitions of restriction are compatible up to natural isomorphism.

To show that transfers are compatible, we will remark that the definitions of restriction and transfer given in \autoref{def:MM-restriction} and \autoref{def:MM-transfer} are adjoint, as are restriction and extension of scalars. By some categorical trickery, transfers will then agree up to natural isomorphism. 

\subsection{Restriction on fixed point subcategories}

In the case where $S$ and $T$ are transitive $G$-sets, the restriction and transfer maps defined above are between categories of the form $\Fun(G/H \times \mathcal{E}G, \mathscr{C})^G$. This is not our working definition of homotopy fixed points, as we were using $\mathscr{C}^{hH} = \Fun(\mathcal{E}H, \mathscr{C})^G$, although we had an explicit equivalence between these two in \autoref{prop:comparison-htpy-fixed-point-definitions}. However in order to relate these restriction and transfer maps to those we build on categories of modules over twisted group rings, we will want to use the equivalence of the twisted group ring module category with homotopy fixed points as in \autoref{prop:htpy-fixed-points-equivalent-module-cat-over-tw-grp-ring}. Thus we will want to rewrite the restriction and conjugation functors of Malkiewich and Merling using our working definition of homotopy fixed points categories.

\begin{definition}\label{def:intermediate-res-subgp} For $f: G/H \to G/K$ associated to a subgroup inclusion $H \subseteq K$, we define the \textit{restriction functor}
\begin{align*}
    f^\ast : \Fun(\mathcal{E}K, \mathscr{C})^K &\to \Fun(\mathcal{E}H, \mathscr{C})^H
\end{align*}
to be precomposition with the canonical functor $\mathcal{E}H \hookto \mathcal{E}K$.
\end{definition}
\begin{proof} We must see these are functorial and well-defined. For restriction, we remark that if $F: \mathcal{E}K \to \mathscr{C}$ is a $K$-equivariant functor, then it is necessarily $H$-equivariant as $H \subseteq K$. Thus to see that the composite $\mathcal{E}H \to \mathcal{E}K \to \mathscr{C}$ is $H$-equivariant, it suffices to observe that the canonical inclusion $\mathcal{E}H \to \mathcal{E}K$ is $H$-equivariant. Verifying that $K$-equivariant natural transformations are send to $H$-equivariant natural transformations is immediate as well.
\end{proof}

\begin{proposition}\label{prop:res-functor-compatibility} The diagram of restriction functors commutes
\begin{equation}\label{eqn:restriction-comparison}
\begin{tikzcd}[ampersand replacement=\&]
    \Fun(G/H\times \mathcal{E}G, \Mod(R))^G \rar["\Phi_H" above, "\sim" below] \& \Fun(\mathcal{E}H, \Mod(R))^H\\
    \Fun(G/K\times \mathcal{E}G, \Mod(R))^G\rar["\sim" above, "\Phi_K" below]\uar["f^\ast" left] \& \Fun(\mathcal{E}K, \Mod(R))^K\uar["f^\ast" right]
\end{tikzcd}
\end{equation}
\end{proposition}
\begin{proof} We begin with $F \in \Fun(G/K\times \mathcal{E}G, \Mod(R))^G$, and see that $\Phi_K(F) \in \Fun(\mathcal{E}K, \Mod(R))^K$ sends $k \mapsto F(eK,k)$, while $f^\ast\Phi_K(F) \in \Fun(\mathcal{E}H, \Mod(R))^H$ sends $h \mapsto F(eK,h)$.

For the other direction, we see $\Phi_H(f^\ast F)$ sends $h \mapsto F(f(eH),h) = F(eK,h)$. Thus the diagram commutes on objects.

Given a $G$-natural transformation $\eta: F \Rightarrow F'$ in $\Fun(G/K \times \mathcal{E}G, \mathscr{C})^G$, we see that $f^\ast \eta$ is of the form $F(f(-),-) \Rightarrow F'(f(-),-)$ by definition. One sees then that $\Phi_H f^\ast \eta$ is the restriction $F(f(eH),-) \Rightarrow F'(f(eH),-)$, which is precisely equal to $F(eK,-) \Rightarrow F'(eK,-)$. One can see that this is the same as $f^\ast \Phi_K \eta$.
\end{proof}

Thus defining restriction on homotopy fixed points for a projection morphism in such a way that it agrees with \autoref{def:MM-restriction} was rather straightforward. Doing an analogous procedure for conjugation morphisms is more difficult, and we will have to use this non-canonical inverse dependent on coset representatives in order to handle this situation. This non-canonical characteristic of our conjugation morphisms will produce agreement only up to natural isomorphism, although this is good enough for our purposes.

\begin{definition}\label{def:conj-on-fixed-points} Let $G/H$ be a transitive $G$-set, let $x\in G$ be arbitrary, and let $g_1, \ldots, g_n$ be a choice of right coset representatives for $G/H$. Then we define a \textit{conjugation functor} $c_x^\ast(g_1, \ldots, g_n)$, depending on the choice of coset representatives, by
\begin{align*}
    c_x^\ast(g_1, \ldots, g_n) : \Fun \left( \mathcal{E} {}^{x}H, \mathscr{C} \right)^{{}^{x}H} &\to \Fun(\mathcal{E}H, \mathscr{C})^H \\
    F &\mapsto \left[ h \mapsto x^{-1}\cdot \left( \til{F}(xh) \right) \right],
\end{align*}
and on morphisms by sending $\eta: F \Rightarrow F'$ to the natural transformation whose component at $y$ is $x^{-1}\cdot \left( \widetilde{\eta}_{x gy} \right)$. We will see that this is a functor via the proof contained in the following proposition.
\end{definition}

\begin{proposition}\label{prop:conj-agree-on-fixed-points} We have that $c_x^\ast(g_1, \ldots, g_n)$ fits into a commutative diagram with the restriction map $c_x^\ast$:
\[ \begin{tikzcd}
    \Fun(G/{}^{x}H \times \mathcal{E}G,\mathscr{C})^G\rar["c_x^\ast" above] & \Fun(G/H\times \mathcal{E}G, \mathscr{C})^G\dar["\sim" right]\\
    \Fun(\mathcal{E} {}^{x}H, \mathscr{C})^{{}^{x}H}\rar[dashed,"{c_x^\ast(g_1, \ldots, g_n)}" below]\uar["\sim" left] & \Fun(\mathcal{E}H, \mathscr{C})^H.
\end{tikzcd} \]
Here the right vertical map is the equivalence from \autoref{prop:comparison-htpy-fixed-point-definitions}, where the left vertical map is the coset-dependence equivalence in \autoref{prop:invs-to-categorical-equivalence-htpy-fixed-pts}.
\end{proposition}
\begin{proof} We verify that the conjugation functor in \autoref{def:conj-on-fixed-points} is indeed a functor and agrees with this diagram above by simply tracing objects and morphisms clockwise through the diagram and verifying that this agrees with what we called conjugation. On objects, we see that we have
\[ \begin{tikzcd}
    {(g{}^{x}H,\widetilde{g}) \mapsto g\cdot \left( \widetilde{F}(g^{-1} \widetilde{g}) \right)}\rar[maps to] & {(gH, \widetilde{g}) \mapsto gx^{-1}\cdot \left( F(xg \widetilde{g}) \right)}\dar[maps to] \\
    F\uar[maps to]\rar[dashed, maps to] & {h\mapsto x^{-1}\cdot \left(\widetilde{F}(xh)\right)}.
\end{tikzcd} \]
On morphisms, we see that
\[ \begin{tikzcd}
    {\til{\eta}_{(g{}^{x}H, \widetilde{g})} := g \cdot \left( \widetilde{\eta}_{g^{-1} \widetilde{g}} \right)} \rar[maps to] & {\left( c_x^\ast\eta \right)_{(gH, \widetilde{g})} = \widetilde{\eta}_{c_x(gH, \widetilde{g})} = \widetilde{\eta}_{gx^{-1} {}^{x}H, \widetilde{g}} }\dar[maps to]\\
    \eta: F \Rightarrow F'\rar[maps to,dashed]\uar[maps to] & {\left( c_x^\ast \eta \right)_{(eH, \widetilde{g})} = \widetilde{\eta}_{x^{-1} {}^{x}H, \widetilde{g}} = x^{-1} \cdot \left( \widetilde{\eta}_{x \widetilde{g}}  \right) }.
\end{tikzcd} \]
\end{proof}

\subsection{Comparison of restriction} We verify that restriction along projection and conjugation morphisms in the orbit category agree in the sense of \cite{MM} and in the sense of restriction of scalars along $\rho_H^K$ and $\gamma^x$.

\begin{lemma}\label{lem:res-functors-compatible} For $f: G/H \to G/K$ a quotient map associated to an inclusion of subgroups $H \subseteq K$, the diagram commutes
\[ \begin{tikzcd}
    \Fun(G/H\times \mathcal{E}G, \Mod(R))^G \rar["\Phi_H" above, "\sim" below] & \Fun(\mathcal{E}H, \Mod(R))^H\rar["\Psi_H" above, "\sim" below] & \Mod_{\twR{H}}\\
    \Fun(G/K\times \mathcal{E}G, \Mod(R))^G\rar["\sim" above, "\Phi_K" below]\uar["f^\ast" left] & \Fun(\mathcal{E}K, \Mod(R))^K\uar["f^\ast" right]\rar["\Psi_K" below, "\sim" above] & \Mod_{\twR{K}}\uar["\Res_H^K" right],\\
\end{tikzcd} \] 
\end{lemma}
\begin{proof} By \autoref{prop:res-functor-compatibility}, the left square commutes, so it suffices to verify the right square agrees.

Suppose $F\in \Fun(\mathcal{E}K, \Mod(R))^K$. Then $\Psi_K(F)$ is the module $M:= F(e) \in \Mod(R)$ with twisted group action given by the morphisms $F(e\to k)$. Under $\Res_H^K$, we simply restrict of the action to the image of morphisms of the form $F(e\to h)$ where $h\in H$.

For the other direction, we see $(f^\ast F)(e) = F(f(e)) = F(e)$, so we obtain the same $R$-module $M = F(e)$. Moreover, for any $e \to h$ in $\mathcal{E}H$, it maps to $e \to h$ in $\mathcal{E}K$. That is, under $\Psi_H(f^\ast F)$ we obtain the same twisted module structure on $M$ as in $\Res_H^K\Psi_K(F)$. Thus the diagram commutes on objects.

Given any natural transformation $\eta: F \Rightarrow F'$ in $\Fun(\mathcal{E}K, \mathscr{C})^K$, we have that $\Psi_K(\eta) = \eta_e : F(e) \to F'(e)$, and that $\Res_H^K$ sends this morphism of $\twR{K}$-modules to the morphism $\eta_e : F(e) \to F'(e)$, viewed as an $\twR{H}$-module homomorphism under restriction of scalars. For the other direction, we see that $f^\ast\eta$ is given by whiskering $\eta$ with $f$, and that $\Psi_H(f^\ast\eta)$ is precisely $(\eta\circ f)_e : (f(e)) \to F'(f(e))$, which is $\eta_e$. Thus the diagram commutes on morphisms.
\end{proof}

\begin{lemma}\label{lem:conjugation} Let $H \subseteq G$ be a subgroup, and $x\in G$ arbitrary, so that we have a conjugation morphism $G/H \to G/{}^{x}H$. When the coset representatives $g_i$ are chosen so that $x=g_1$, the diagram commutes up to natural isomorphism 
\begin{equation}\label{eqn:conjugation}
\begin{aligned}
    \begin{tikzcd}[ampersand replacement=\&]
    {\Fun(G/{}^{x}H \times \mathcal{E}G,\mathscr{C})^G}\dar["c_x^\ast"]\rar["\sim"] \& {\Fun(\mathcal{E} {}^{x}H, \Mod(R))^{{}^{x}H}} \dar["{c_x^\ast(g_1, \ldots, g_n)}" left]\rar["\sim"] \& {\Mod_{\twR{{}^{x}H}}} \dar["{c_x^\ast}"]\\
    {\Fun(G/H\times \mathcal{E}G, \mathscr{C})^G}\rar["\sim" below] \& {\Fun(\mathcal{E}H, \Mod(R))^H}\rar["\sim" below] \& {\Mod_{\twR{H}} }
\end{tikzcd}
\end{aligned}
\end{equation}

where $c_x^\ast$ is restriction of scalars along the ring homomorphism $\twR{H} \to \twR{{}^{x}H}$.
\end{lemma}

\begin{proof} By \autoref{prop:conj-agree-on-fixed-points}, the left square commutes up to natural isomorphism, so it suffices to check the right square. Let $F \in \Fun(\mathcal{E} {}^{x}H, \Mod(R))^{{}^{x}H}$. Then its image in $\Mod_{{}^{x}H}[{}^{x}H]$ is given by the module $M:= F(e)$ equipped with the following $\twR{{}^{x}H}$-module structure
\begin{align*}
    \twR{{}^{x}H} &\to \End_\Ab(M) \\
    ry &\mapsto \left( M \xto{F(e,y)} M \xto{r\cdot-} M \right).
\end{align*}
Under restriction of scalars, it is sent to the $\twR{H}$-module $M$, with action given by
\begin{equation}\label{eqn:RHH-action-1}
\begin{aligned}
    \twR{H} &\to \End_\Ab(M) \\
    rh &\mapsto \left( M \xto{F(e,xhx^{-1})} M \xto{r\cdot-} M \right).
\end{aligned}
\end{equation}

Conversely, let $\Psi\in \Fun(\mathcal{E}H, \Mod(R))^H$ be $c_x^\ast(g_1, \ldots, g_n)(F)$, that is,
\begin{align*}
    \Psi(h) = x^{-1}\til{F}(xh),
\end{align*}
where $F$ is extended to a functor $\til{F}$ on $\mathcal{E}G$ by selection of coset representatives. A priori, there is no good way to relate $\til{F}(xh)$ to $\til{F}(x)$, since one must write $xh = yg_i$ for some $y \in {}^{x}H$ and $g_i$ a coset representative, and then express $\til{F}(xh) := F(y)$. However, we remark that \textit{we can choose our coset representatives}. To that end, since $|G:{}^{x}H|\ge 2$, we may select $g_1 = x$, and let $g_2, \ldots, g_n$ be an arbitrary selection of representatives for the remaining cosets. This has the following advantage: we can write
\begin{align*}
    xh &= (xhx^{-1})x,
\end{align*}
and then one sees that since $xhx^{-1} \in {}^{x}H$, and $x$ is a coset representative, we have that $\til{F}(xh) = F(xhx^{-1})$ for all $h \in H$. In particular, one remarks that $\til{F}(x) = F(e)$. We therefore see that $\Psi$ takes the following form:
\begin{align*}
    \Psi(h) = x^{-1}F(xhx^{-1}).
\end{align*}
We see that $\Psi$ is sent to the abelian group $\Psi(e) = x^{-1}F(xex^{-1}) = x^{-1}F(e)$, and since $x^{-1}F(e)=F(e)=:M$ by assumption, we have that $\Psi \mapsto M$. Thus \autoref{eqn:conjugation} produces the same abelian group.

To see that the $\twR{H}$-module structures on $M$ are the same, we see that the one produced by $\Psi$ yields

\begin{equation}\label{eqn:RHH-action-2}
\begin{aligned}
     \twR{H} &\to \End_\Ab(M) \\
    rh &\mapsto \left( M \xto{\Psi(e,h)} M \xto{r\cdot-} \right).
\end{aligned}
\end{equation}

We must see that Equations~\autoref{eqn:RHH-action-1} and \autoref{eqn:RHH-action-2} produce the same $\twR{H}$-module structure on $M$. It suffices to see that $\Psi(e,h)$ and $F(e,xhx^{-1})$ agree as abelian group homomorphisms for every $h\in H$. We see that $\Psi(e,h)$ is the image under $\Psi$ of the unique map $e\to h$ in $\mathcal{E}H$, and that, as a morphism in $\Mod(R)$, it is a map between $x^{-1}\til{F}(xe) = x^{-1}M$, and $x^{-1}\til{F}(xh) = x^{-1}F(xhx^{-1})$. That is, it is precisely the map $x^{-1} \left( F(e,xhx^{-1}) \right)$. One sees therefore that the diagram of abelian group homomorphisms commutes
\[ \begin{tikzcd}
    M\rar["x^{-1}\cdot-" above]\ar[dr,"{\Psi(e,h)}" below left] & M\dar["{F(e,xhx^{-1})}" right]\\
     & M.\\
\end{tikzcd} \]
As $x^{-1}$ is the identity as an abelian group homomorphism by definition, we see that $\Psi(e,h) = F(e,xhx^{-1})$ in $\Hom_\Ab(M, F(xhx^{-1}))$, and therefore that \autoref{eqn:RHH-action-1} and \autoref{eqn:RHH-action-2} produce the same $\twR{H}$-action on $M$.

We now verify that \autoref{eqn:conjugation} agrees on morphisms. Let $\eta: F \Rightarrow F'$ be an arbitrary ${}^{x}H$-equivariant natural transformation between functors $F,F' \in \Fun(\mathcal{E}{}^{x}H, \Mod(R))^{{}^{x}H}$. Then we have that $c_x^\ast(g_1, \ldots, g_n)(\eta)$ is given by, for $h \in \mathcal{E}H$, the component 
\[ x^{-1} \cdot\left( \widetilde{\eta}_{xh} \right) : x^{-1} \cdot \left( \til{F}(xh) \right) \to x^{-1} \cdot \left( \widetilde{F'}(xh)\right).\]
Under the equivalence $\Fun(\mathcal{E}H, \Mod(R))^H \xto{\sim} \Mod_{\twR{H}}$, we send this to its component at the identity $e\in H$, which is of the form (by recalling that $\widetilde{F}(x) = F(e)$ and then that $\widetilde{\eta}_x = \eta_e$):
\begin{align*}
    x^{-1}\cdot \left( \eta_e \right) : x^{-1} \left( F(e) \right).
\end{align*}
Conversely, under the functor $\Fun(\mathcal{E}{}^{x}H, \Mod(R))^{{}^{x}H} \to \Mod_{\twR{{}^{x}H}}$, we see that $\eta$ is sent to $\eta_e : F(e) \to F'(e)$. Under the conjugation functor $\gamma_x^\ast:\Mod_{\twR{{}^{x}H}} \to \Mod_{\twR{H}}$, this is sent to $\eta_e$, viewed as an $\twR{H}$-module homomorphism. This isn't equal on the nose to the morphism $x^{-1}\eta_e$, however we remark that, by post-composition with the $\twR{H}$-module isomorphism $x^{-1}\cdot -$, these morphisms agree. Thus the diagram commutes up to natural isomorphism.
\end{proof}

\subsection{Comparison of transfers}

\begin{lemma}\label{lem:tr-functors-compatible} Let $f: G/H \to G/K$ be a morphism in the orbit category.
\begin{enumerate}
    \item If $f$ is a projection morphism for $H \subseteq K$, then $f_!$ agrees with $\Tr_H^K$ up to natural isomorphism.
\[ \begin{tikzcd}
    \Fun(G/H \times \mathcal{E}G, \mathcal{P}(R))^G \rar["\sim"]\dar["f_!" left] & \mathcal{P}(R_\theta[H])\dar["\Tr_H^K" right]\\
    \Fun(G/K \times \mathcal{E}G, \mathcal{P}(R))^G\rar["\sim" below] & \mathcal{P}(R_\theta[K]).
\end{tikzcd} \]

    \item If $f$ is a conjugation morphism, that is, $K = {}^{g}H$, then $f_!$ agrees with $\gamma^g_!$ up to natural isomorphism.
    \[ \begin{tikzcd}
    \Fun(G/H \times \mathcal{E}G, \mathcal{P}(R))^G \rar["\sim"]\dar["f_!" left] & \mathcal{P}(R_\theta[H])\dar["\gamma^g_!" right]\\
    \Fun(G/{}^{g}H\times \mathcal{E}G, \mathcal{P}(R))^G\rar["\sim" below] & \mathcal{P}(R_\theta[{}^{g}H]).
\end{tikzcd} \]

\end{enumerate}
\end{lemma}
This follows as a particular case of the following, more general proposition.

\begin{proposition}\label{prop:square-proposition} Suppose we have a diagram of functors which commutes (up to natural isomorphism)
\[ \begin{tikzcd}
    {A}\rar["p^\ast" above]\dar["r^\ast" left] & {B}\dar["s^\ast" right]\\
    {C}\rar["q^\ast" below] & {D}'
\end{tikzcd} \]
so that each functor admits a left adjoint, which we decorate with a lower shriek, and moreover assume that $p^\ast$ and $q^\ast$ are equivalences of categories. Then the diagram
\[ \begin{tikzcd}
    {A}\rar["p^\ast"] & {B}\\
    {C}\uar["r_!" left]\rar["q^\ast" below] & {D}\uar["s_!" right]
\end{tikzcd} \]
commutes up to natural isomorphism.
\end{proposition}
\begin{proof} We first see that the intermediate diagram
\[ \begin{tikzcd}
    {A}\rar["p^\ast"]\dar["r^\ast" left] & {B}\dar["s^\ast" right]\\
    {C} & {D}\lar["q_!" below]
\end{tikzcd} \]
commutes up to natural isomorphism. Indeed, this is witnessed by the following composite (where $\epsilon$ is the counit associated to the adjunction $q_! q^\ast \to \id$, which is a natural isomorphism as $q^\ast$ is an equivalence of categories):
\begin{align*}
    q_!(s^\ast p^\ast) &\cong q_!(q^\ast r^\ast) \xto{\epsilon} r^\ast.
\end{align*}

We then claim that $r_! \cong p_! s_! q^\ast$. This follows from uniqueness of adjoints, as $r^\ast \cong q_! s^\ast p^\ast$, and $q_! s^\ast p^\ast$ admits a left adjoint, given by $p_! s_! q^\ast$ (since $g$ is an equivalence, we have that $q^\ast \dashv q_! \dashv q^\ast$). Combining this with the natural isomorphism $p^\ast p_! \xto{\sim} \id$ given by $p^\ast$ being an equivalence of categories, we have that
\begin{align*}
    p^\ast r_! \cong p^\ast p_! s_! q^\ast \xto{\sim} s_! q^\ast.
\end{align*}
\end{proof}

\subsection{Proof of \autoref{thm:equivalence-Mackey-functors}}

We first state a concise corollary of the work in the previous few sections.

\begin{corollary}\label{cor:res-tr-conj-agree-projective} Let $G$ be finite, and $|G|^{-1}\in R$. Then restriction, transfer and conjugation for module categories of twisted group rings agree with restriction, transfer and conjugation for homotopy fixed points up to natural isomorphism. Moreover, this compatibility holds on the subcategories of finitely generated projective modules.
\end{corollary}

After applying $K$-theory, we have that the equivalence of categories $\mathcal{P}(R)^{hH} \simeq \mathcal{P}(\twR{H})$ from \autoref{cor:descent-to-proj-modules-subcat} yields an isomorphism of spectra $K(\mathcal{P}(R)^{hH}) \simeq K(\mathcal{P}(\twR{H}))$. In particular, our definitions of restriction, transfer, and conjugation agree on these spectra, since they agreed up to natural isomorphism on the module categories. As restriction, transfer, and conjugation endowed $G/H \mapsto K(\mathcal{P}(R)^{hH})$ with the data of a $\mathcal{B}_G^\Wald$-module, by the compatibility we have proven, this implies that $G/H \mapsto K(\mathcal{P}(\twR{H}))$ is a $\mathcal{B}_G^\Wald$-module, and is moreover isomorphic to the $\mathcal{B}_G^\Wald$-module $G/H \mapsto K(\mathcal{P}(R)^{hH})$.

Finally, under the Quillen equivalence between module categories over $\mathcal{B}_G$ and $\mathcal{B}_G^\Wald$ (\autoref{cor:GB-and-GBWald-equiv}), after possibly modifying by an equivalence, we have that $G/H \mapsto K(\mathcal{P}(\twR{H}))$ and $G/H \mapsto K(\mathcal{P}(R)^{hH})$ are equivalent as $\mathcal{B}_G$-modules. This proves \autoref{thm:equivalence-Mackey-functors}. 

\subsection{Mackey functors from spectral Mackey functors, and \autoref{cor:KGn-computes-pi-n-KGR}}\label{subsec:mf-from-smf}

From our previous work, we have obtained a spectral Mackey functor $\mathbf{K}_G(R)$. By taking homotopy groups, we have hinted at the ability to obtain Mackey functors at each level. We will now make this procedure explicit.

Given any spectrally enriched category $\mathcal{A}$, and any lax monoidal functor $F: \Sp \to \Ab$, we can define a new category $F_\bullet \mathcal{A}$, which is a pre-additive category on the same objects. The homs in $F_\bullet \mathcal{A}$ are obtained as follows: for any $a,a'\in \mathcal{A}$,we define
\begin{align*}
    \Hom_{F_\bullet \mathcal{A}}(a,a') := F_\bullet \Hom_{\mathcal{A}}(a,a').
\end{align*}
The assumption that $F$ is lax monoidal implies that the resulting category has a well-defined composition enriched over $\Ab$. For a more general version of this statement, see \cite[Proposition~2.11]{BO15}.

\begin{example} As $\pi_0$ is lax monoidal, we may apply $\left( \pi_0 \right)_\bullet$ to the spectral Burnside category in order to recover the ordinary Burnside category
\[ \left( \pi_0 \right)_\bullet \mathcal{B}_G = B_G.\]
\end{example}

Given any spectrally enriched functor $\mathcal{A} \to \Sp$, by applying $\left( \pi_0 \right)_\bullet$, we obtain an additive functor $\left( \pi_0 \right)_\bullet \mathcal{A} \to \Ho(\Sp)$, valued in the homotopy category of spectra. Post-composing with any stable homotopy group functor $\pi_n$ produces a module over $\left( \pi_0 \right)_\bullet \mathcal{A}$.

In particular, if $\Phi: \mathcal{B}_G \to \Sp$ is any spectral Mackey functor, and $n\ge 0$, then we obtain the $n$th \textit{homotopy Mackey functor}, denoted $\und{\pi}_n \Phi$, via the following composite
\[ B_G \cong \left( \pi_0 \right)_\bullet \mathcal{B}_G \xto{\left( \pi_0 \right)_\bullet \Phi} \Ho(\Sp) \xto{\pi_n} \Ab.\]
This is due to \cite[Proposition~7.6]{BO15}. As an immediate example, we observe that $\underline{K}_n^G(R) = \und{\pi}_n \mathbf{K}_G(R)$, yielding \autoref{cor:KGn-computes-pi-n-KGR}.

\section{Families of Mackey functors}\label{sec:examples}

To wrap up, we provide a few particular examples and applications of our work constructing the Mackey functors $\underline{K}_n^G(R)$ for a $G$-ring $R$. In particular, we can generalize these results to provide a family of Mackey functors $E_n(R)$ for any suitable invariant $E$ of rings. We explain this in detail in the following section. 

\subsection{Mackey functors arising from homotopy invariants of rings}\label{subsec:MF-additive-invariants}

Once we had constructed the functors $\Tr_K^H$, $\Res_K^H$ and $\gamma^g_!$ between the exact categories of finitely generated projective modules over twisted group rings, and we had proved they satisfied axioms related to Mackey functors up to natural isomorphism, the construction of the Mackey functor $\underline{K}_n^G(R)$ was basically immediate. In particular we relied on two key aspects of the functor $K_n$:
\begin{enumerate}[(1)]
    \item If $f: R \to S$ exhibits $S$ as a finitely generated projective $R$-module, then restriction of scalars along $f$ descends to a functor $f^\ast: \mathcal{P}(S) \to \mathcal{P}(R)$, and extension of scalars is a functor $f_!: \mathcal{P}(R) \to \mathcal{P}(S)$ (even without these conditions). These in turn induce abelian group homomorphisms $f^\ast: K_n(S) \to K_n(R)$ and $f_! : K_n(R) \to K_n(S)$.
    \item If $F,G : \mathscr{C} \to \mathscr{D}$ are naturally isomorphic exact functors of exact categories, then $K_n(F) = K_n(G)$.
\end{enumerate}

Denote by $\ExCat$ the 2-category whose objects are exact categories, and whose morphisms are given by exact functors. Let $\Sp_{\ge 0}$ denote the category of connective spectra (i.e. infinite loop spaces). We see that condition (1) can be thought of as a consequence of the fact that we may view $K$-theory as a 1-functor $\ExCat \to \Sp_{\ge 0}$. Consequence (2) says that $K$-theory sends natural isomorphisms to homotopies.

\begin{terminology} We refer to any functor $E: \ExCat \to \Sp_{\ge 0}$ sending natural isomorphisms to homotopies as a \textit{homotopy invariant} of exact categories. Denoting by $E_n: \ExCat \to \Ab$ the composite functor $\pi_n\circ E$, we have that $E_n$ satisfies the conditions (1) and (2) listed above. Thus any such functor will provide a family of Mackey functors via an analogous proof to that of \autoref{thm:Mackey-func}. For any ring $R$, denote by $E(R) := E(\mathcal{P}(R))$ the space given by evaluating $E$ on the exact category of finitely generated projective $R$-modules, and denote by $E_n(R):= \pi_n E(R)$ its $n$th homotopy group.
\end{terminology}

\begin{corollary}\label{cor:func-from-excat-is-mackey-functor} Let $R$ be any $G$-ring, where $G$ is a finite group whose order is invertible over $R$. Let $E$ be any homotopy invariant on exact categories, and let $n\ge 0$ be arbitrary. Then we have that
\begin{align*}
    G/H \mapsto E_n \left( R_\theta[H] \right)
\end{align*}
is a Mackey functor, where restriction and transfer are induced by restriction and extension of scalars along exact categories of finitely generated projective modules over twisted group rings.
\end{corollary}

\begin{example} We may let $E$ be any one of $\HH, \THH,\TC,\TP,\TR$ in Corollary~\ref{cor:func-from-excat-is-mackey-functor}, yielding many new Mackey functors such as $\THH_n(R_\theta[H])$, the topological Hochschild homology of twisted group rings. The structure of $\THH_n(R_\theta[H])$ was explored in the thesis of Daniel Vera \cite{Vera-thesis}, although this Mackey functor structure was not known. 
\end{example}

The structure of the Mackey functors associated to these homotopy invariants, as well as their relation to trace maps, will be explored in a later paper.

\subsection{The Mackey functors of $K$-theory of fixed subrings of Galois extensions}\label{subsec:MF-Galois}
Given a finite Galois extension of fields $G = \Gal(L/k)$, we may consider $L$ as a $G$-ring under the natural Galois group action. It is a classical example that the assignment $G/H \mapsto K_n(L^H)$ is a Mackey functor (see, for example \cite[53.10]{Thevenaz}). We can prove a new proof for any Galois extension of rings, relying on Galois descent, which follows as an immediate corollary of \autoref{thm:Mackey-func}.

Let $R \to S$ be a Galois extension of rings (see \cite[p.~396]{Auslander-Goldman} for the original definition) with Galois group $G = \Gal(S/R)$, and denote by $\theta$ the action of $G$ on $S$. We remark that an $S_\theta[G]$ module is an $S$-module equipped with a semilinear action of the Galois group. This is the same as an $S$-module equipped with \textit{Galois descent data}. We recall that modules with descent data are equivalent to modules over the fixed subring, explicitly we have an equivalence of categories $\Mod(S_\theta[H]) \simeq \Mod(S^H)$ for any subgroup $H \subseteq G$, and it is easy to see that this equivalence descends to finitely generated projective modules.\footnote{An equivalence of categories can easily be seen to preserve projective objects --- if a morphism admits left lifting with respect to epimorphisms in the target, then the functor exhibiting an equivalence preserves epimorphisms (since it is a left adjoint perhaps after promoting to an adjoint equivalence), so the image of this morphism admits left lifting with respect to all epimorphisms in the essential image, which is the entire target category by equivalence.}

\begin{proposition}\label{prop:k-theory-galois-rings} Let $R \to S$ be a Galois extension of rings, and assume that the order of the Galois group $G$ is invertible over $R$. Then for any $n\ge 0$ we have that
\begin{align*}
    G/H \mapsto K_n(S^H)
\end{align*}
is a Mackey functor, where restriction and transfer come from restriction and extension of scalars between fixed subrings.
\end{proposition}

\begin{corollary} Under the conditions of \autoref{prop:k-theory-galois-rings}, we have that
\begin{align*}
    G/H \mapsto \THH_n(S^H)
\end{align*}
is a Mackey functor for any $n\ge 0$ (and an analogous statement is true by replacing $\THH$ with any homotopy invariant $E$).
\end{corollary}

\subsection{The Mackey functor of $K$-theory of endomorphism rings}\label{subsec:MF-endo-rings}

Suppose that $\theta: G \to \Aut_\Ring(R)$ is a group action on a ring $R$. Then let $R^G \subseteq R$ denote the subring of $G$-fixed points under this action. We may define a ring homomorphism
\begin{equation}\label{eqn:gamma-map}
\begin{aligned}
    R_\theta[G] &\to \End_{R^G}(R) \\
    rg &\mapsto \left( t \mapsto r\theta_g(t) \right).
\end{aligned}
\end{equation}

In general, we shouldn't expect this map to be injective or surjective. However by a classical theorem of Auslander, we can state a sufficient condition for this ring homomorphism to be an isomorphism. In the following theorem we retain our ongoing assumption that $G$ is finite and that $|G|^{-1}\in R$

\begin{theorem}\label{thm:auslander} \cite[Proposition~3.4]{Auslander} If $R$ is a normal domain, and $R$ is unramified in codimension one over $R^G$, then the ring homomorphism in \autoref{eqn:gamma-map} is an isomorphism.
\end{theorem}

For a detailed proof of this statement, we refer the reader to \cite[Chapter~5.2]{Leuschke}. Via our result in \autoref{thm:Mackey-func}, we can then prove that the algebraic $K$-groups of such endomorphism rings admit the structure of a Mackey functor. In particular the restrictions and transfers would arise from passing through the isomorphism with twisted group rings. We can actually say a bit more, as we can show that there is a natural choice of restriction and transfer coming intrinsically from endomorphism rings.

Suppose that $G$ and $R$ satisfy the conditions of \autoref{thm:auslander}. Let $H \subseteq K$ be a subgroup. Then we have that $R^K \subseteq R^H$. In particular, any endomorphism of $R$ which is fixed over $R^H$ is also fixed over $R^K$. Thus there is a natural forgetful ring homomorphism $\End_{R^H}(R) \to \End_{R^K}(R)$ fitting into the diagram
\[ \begin{tikzcd}
    R_\theta[H]\rar\dar & \End_{R^H}(R)\dar\\
    R_\theta[K]\rar & \End_{R^K}(R).
\end{tikzcd} \]
Similarly if $H \subseteq G$ is a subgroup and $g\in G$ any element, there is a natural ring homomorphism of endomorphism rings $\End_{R^H}(R) \to \End_{R^{{}^{g}H}}(R)$, given by sending $\phi \mapsto \theta_g \circ \phi \circ \theta_{g^{-1}}$. We can easily see this fits into a diagram
\[ \begin{tikzcd}
    R_\theta[H]\rar\dar["\gamma^g" left] & \End_{R^H}(R)\dar\\
    R_\theta[{}^{g}H]\rar & \End_{R^{{}^{g}H}}(R).
\end{tikzcd} \]

\begin{proposition}\label{prop:Mackey-functor-unramified-ring-ext} Let $G$ be a finite group acting on a normal domain $R$, so that $|G|^{-1}\in R$ and $R$ is unramified of codimension one over $R^G$. Then for any $n\ge 0$, we have that
\begin{align*}
    G/H \mapsto K_n\left( \End_{R^H}(R)\right)
\end{align*}
is a Mackey functor, where
\begin{itemize}
    \item transfer $\Tr_H^K$ for a subgroup $H \subseteq K$ is induced by extension of scalars along the inclusions $\End_{R^H}(R) \to \End_{R^K}(R)$
    \item restriction $\Res_H^K$ is induced by restriction of sclars along $\End_{R^H}(R) \to \End_{R^K}(R)$
    \item conjugation is induced by extension of scalars along  the ring homomorphism $\End_{R^H}(R) \to \End_{R^{{}^{g}H}}(R)$ sending $\phi \mapsto \theta_g\circ \phi\circ \theta_{g^{-1}}$.
\end{itemize}
\end{proposition}

\begin{example} Under the conditions of \autoref{prop:Mackey-functor-unramified-ring-ext}, we have that
\begin{align*}
    G/H \mapsto \THH_n(\End_{R^H}(R))
\end{align*}
is a Mackey functor (where we may also replace $\THH$ by any homotopy invariant).
\end{example}

\appendix

\section{Bimodule structures}\label{sec:bimodule}

Let $F: \Mod(R)\to \Mod(S)$ denote a right exact coproduct-preserving functor between left module categories. Then $F(R)$ inherits a right $R$-module structure, as described in \cite{Watts}. Let $m\in M$, and consider the multiplication morphism
\begin{align*}
    \omega_m : R &\to M \\
    r &\mapsto rm.
\end{align*}
Applying $F$, this gives a map $F\omega_m : FR \to FM$ for all $m\in M$. Letting $m$ vary, we have a morphism
\begin{align*}
    FR \times M &\to FM \\
    (x,m) &\mapsto F\omega_m(x).
\end{align*}
Letting $M=R$, we have a map of the form
\begin{align*}
    FR \times R &\to FR \\
    (x,r) &\mapsto F\omega_r(x).
\end{align*}
We may verify this defines a right $R$-module structure on $FR$.

This holds in general for right exact functors preserving coproducts. In our case, we will be interested in building up an understanding of these structures for the case where $F$ is extension or restriction of scalars along the ring maps $\rho_H^K$ and $\gamma^x$, with the ultimate goal of comparing the right $\twR{H}$-module structures on $ \bigoplus_{i=1}^n \Tr_{J\cap {}^{x} K}^J \gamma^x_! \Res_{J^x \cap K}^K \left( \twR{H} \right)$ and $\Res_J^H \Tr_K^H \twR{H}$.

\begin{proposition} Let $\omega_r : R \to R$ be right multiplication by $r$ viewed as a left $R$-module homomorphism, and let $f: R \to S$. Then $f_! \omega_r$ is of the form
\begin{align*}
    f_! \omega_r : S \otimes_R R &\to S \otimes_R R \\
    (s' \otimes r') &\mapsto (s' \otimes r')(1 \otimes r) = s \otimes r'r = s f(r'r).
\end{align*}
So we have that $f_!\omega_r = \omega_{f(r)}$ is right multiplication by $f(r)$.
\end{proposition}

\begin{proposition} Let $\omega_r : R \to R$ be right multiplication by $r$ viewed as a left $R$-module homomorphism, and let $g: T \to R$. Then $g^\ast \omega_r$ is of the form
\begin{align*}
    g^\ast\omega_r: g^\ast R &\to g^\ast R \\
    r' &\mapsto r'\cdot r,
\end{align*}
where $-\cdot r$ is viewed as a $T$-module homomorphism.
\end{proposition}

This leads us to the very rough rule that $f_! \omega_r = \omega_{f(r)}$, and $f^\ast \omega_r = \omega_r$, where one needs to be careful about the ambient context.

\begin{proposition} Let $\rho_H^K: \twR{H} \to \twR{K}$, and let $H\backslash K = Hy_1, \ldots, Hy_n$ be a choice of right coset representatives. Let $\gamma^x: \twR{H} \to \twR{{}^{x}H}$.
\begin{enumerate}
    \item Consider extension of scalars $(\rho_H^K)_!: \Mod(\twR{H}) \to \Mod(\twR{K})$. We have that the right $\twR{H}$-module structure on $\left( \rho_H^K \right)_! (\twR{H})$ is given by right multiplication on $\twR{K}$.

    \item Consider restriction of scalars $\left( \rho_H^K \right)^\ast : \Mod(\twR{K}) \to \Mod(\twR{H})$. Then the right $\twR{K}$-module structure on $\left( \rho_H^K \right)^\ast (\twR{H})$ is given by right multiplication by $\twR{K}$ viewed as an $\twR{H}$-module isomorphism. Explicitly, it is a map of the form
    \begin{align*}
        \oplus_i \twR{H}y_i \times \twR{K} &\to \oplus_i \twR{H} y_i \\
        \left( \sum_i r_i h_i y_i , rk \right) &\mapsto \left( \sum_i r_i\phi_{h_i}(r) h_i k \right),
    \end{align*}
    where we express $\sum_i r_i\phi_{h_i}(r) h_i k = \sum_i r_i' h_i' y_i$ as a new sum in the $y_i$'s.

    \item Consider extension of scalars $(\gamma^x)_! : \twR{H} \to \twR{{}^{x}H}$. Then the right module structure on $\gamma^x_!(\twR{H})$ is given by right multiplication by $\twR{H}$ under the map $\gamma^x$. Explicitly it is
    \begin{align*}
         \twR{{}^{x}H} \times  \twR{H} &\to \twR{{}^{x}H} \\
         (ry, rh) &\mapsto ry \cdot \gamma^x(rh).
    \end{align*}
\end{enumerate}
\end{proposition}

\begin{corollary} Let $J,K \subseteq H$, and let $J\backslash H = \cup_{i=1}^r Jy_i$, and let $K\backslash H = \cup Kz_j$. Then the right $\twR{K}$-module structure on $\Res_J^H \Tr_K^H \twR{K} \cong \oplus_{i=1}^r \twR{J} 1_R y_i$ is given by viewing $rk \in \twR{H}$, and writing $rk = rjy_\ell$ for some $y_\ell$ and $j\in J$, and then considering the map
\begin{align*}
    \oplus_{i=1}^r \twR{J} 1_R y_i \times \twR{K} &\to \oplus_{i=1}^r \twR{J} 1_R y_i \\
    \left( \sum r'j' y_i , rjy_\ell\right) &\mapsto \left( \sum r' \phi_{j'}(r) y_i j y_\ell \right).
\end{align*}
\end{corollary}

\begin{corollary} Let $J,K \subseteq H$, and let $x$ be a choice of double coset representative for $J\backslash H /K$. Let $J^x \cap K \backslash K =\cup_i (J^x \cap K)\beta_i$. Then the right module structure of $\twR{K}$ on 
\begin{align*}
    \Tr_{J\cap {}^{x} K}^J \gamma^x_! \Res_{J^x \cap K}^K \left( \twR{K} \right)
\end{align*}
is given by first writing $rk = rx^{-1}j x\beta_\ell$ for some $\beta_\ell$, and right multiplying. Then after $\gamma^x_!$, it is given by right multiplication by $\gamma^x(rx^{-1} j x \beta_\ell)$, which is $\phi_x(r) j x\beta_\ell x^{-1}$, then finally transferring, i.e. multiplying through by this as an $\twR{J}$-module homomorphism on the right.
\end{corollary}

\begin{proof}[Proof of \autoref{prop:bimodule-homomorphism}]

Let $\sum r_i k_i \in \twR{K}$ be arbitrary, and let $\omega_{\sum r_i k_i}$ denote the right multiplication by this element, viewed as a left $\twR{K}$-module homomorphism. Then we see that, for a given $x$, we have
\begin{align*}
    \Tr_{J\cap {}^{x} K}^J \gamma^x_! \Res_{J^x \cap K}^K \omega_{\sum r_i k_i}  &= \omega_{\sum \phi_x(r_i) x k_i x^{-1}},
\end{align*}
by first viewing $\omega_{\sum r_i k_i}$ as an $\twR{J^x \cap K}$-module homomorphism under restriction, then extending to obtain $\omega_{\sum \phi_x(r_i) x k_i x^{-1}}$ as an $\twR{J\cap {}^{x} K}$-module homomorphism, and finally restricting to view $\omega_{\sum \phi_x(r_i) x k_i x^{-1}}$ as an $\twR{J}$-module homomorphism.

We see that
\begin{align*}
    \Res_J^H \Tr_K^H \omega_{\sum r_i k_i} = \omega_{\sum r_i k_i},
\end{align*}
by extending along $\rho_K^H$ and restricting along $\rho_J^H$.

 For the sake of notation, let
\begin{align*}
    P &:= \oplus \Tr_{J\cap {}^{x} K}^J \gamma^x_! \Res_{J^x \cap K}^K \\
    Q &:= \Res_J^H \Tr_K^H.
\end{align*}
Let $\epsilon : P(\twR{K}) \xto{\sim} Q(\twR{K})$ denote the isomorphism of $\twR{J}$-modules given in \autoref{cor:isom-base-rings}. Then it suffices to verify that $F\omega_{\sum r_i k_i} \epsilon = \epsilon G\omega_{\sum r_i k_i}$.

Recall that
\begin{align*}
    P(\twR{K}) &= \bigoplus_{i=1}^n \bigoplus_{\ell=1}^{r_i} \twR{J} x_i \beta_{i,j} x_i^{-1} \\
    Q(\twR{K}) &= \bigoplus_{i=1}^n \bigoplus_{\ell=1}^{r_i} \twR{J} x_i \beta_{i,j}.
\end{align*}
Let $r_\ell j_\ell x_\ell \beta_{\ell,a} x_\ell^{-1}$ be an arbitrary summand in $P(\twR{K})$. Then we consider the following (a priori noncommutative) diagram
\begin{equation}\label{eqn:bimodule-diagram}
\begin{aligned}
    \begin{tikzcd}[ampersand replacement=\&]
    {P(\twR{K})}\rar["F\omega_{\sum r_i k_i}" above]\dar["\epsilon" left] \& {P(\twR{K})}\dar["\epsilon" right]\\
    {Q(\twR{K})}\rar["G\omega_{\sum r_i k_i}" below] \& {Q(\twR{K})}.\\
\end{tikzcd}
\end{aligned}
\end{equation}
Tracing through where the element $r_\ell j_\ell x_\ell \beta_{\ell,a} x_\ell^{-1}$ maps, we see that
\[ \begin{tikzcd}[column sep=7em]
r_\ell j_\ell x_\ell \beta_{\ell,a} x_\ell^{-1}\rar[maps to,"{-\cdot \sum_i \phi_{x_\ell}(r_i) x_\ell k_i x_\ell^{-1}}"]\dar["-\cdot x_\ell" left] & \sum_i r_\ell \phi_{j_\ell x_\ell \beta_{\ell,a}}(r_i) j_\ell x_\ell \beta_{\ell,a} k_i x_\ell^{-1}\dar["-\cdot x_\ell" right]\\
    r_\ell j_\ell x_\ell \beta_{\ell,a}\rar["-\cdot \sum_i r_i k_i" below] & \sum_i r_\ell \phi_{j_\ell x_\ell \beta_{\ell,a}}(r_i) x_\ell \beta_{\ell,a} k_i.\\
\end{tikzcd} \]
Thus since $r_\ell j_\ell x_\ell \beta_{\ell,a} x_\ell^{-1}$ was arbitrary, we have that \autoref{eqn:bimodule-diagram} commutes. This implies that $\epsilon$ is a right $\twR{J}$-module homomorphism.
\end{proof}

\bibliographystyle{amsalpha}
\bibliography{citations.bib}{}
\newpage

\end{document}